\documentclass[a4paper]{amsart}

\title[Scissors automorphism groups II: Solomon--Tits theorems]{Scissors automorphism groups II: \\ Solomon--Tits theorems}

\author[A. Kupers]{Alexander Kupers}
\address{Department of Computer and Mathematical Sciences, University of Toronto Scarborough, 1265 Military Trail, Toronto, ON M1C 1A4, Canada}
\email{a.kupers@utoronto.ca}

\author[E. Lemann]{Ezekiel Lemann} 
\address{Institute of Mathematics, University of Aberdeen, Aberdeen AB24 3UE, UK}
\email{ezekiel.lemann@abdn.ac.uk}

\author[C. Malkiewich]{Cary Malkiewich}
\address{Department of Mathematics and Statistics, Binghamton University, PO Box 6000, Binghamton, NY 13902}
\email{cmalkiew@binghamton.edu}

\author[J. Miller]{Jeremy Miller} 
\address{Mathematical Sciences Building, Purdue University 150 N University St, West Lafayette, IN 47907}
\email{jeremykmiller@purdue.edu}

\author[R. J. Sroka]{Robin J. Sroka} 
\address{Mathematisches Institut, Universität Münster, Einsteinstrasse 62, 48149 Münster, Germany}
\email{robinjsroka@uni-muenster.de}

\usepackage[backend=bibtex, style=alphabetic, hyperref, backref, backrefstyle=none, maxbibnames=99, sorting=nyt, safeinputenc, giveninits=true]{biblatex}
\AtEveryBibitem{\clearlist{language}}
\RequirePackage{doi}

\renewbibmacro*{doi+eprint+url}{
	\iftoggle{bbx:doi}
	{\printfield{doi}}
	{}
	\newunit\newblock
	\ifboolexpr{togl {bbx:eprint} and test {\iffieldundef{doi}}}
	{\usebibmacro{eprint}}
	{}
	\newunit\newblock
	\ifboolexpr{togl {bbx:url} and test {\iffieldundef{doi}}  and test {\iffieldundef{eprint}}}
	{\usebibmacro{url+urldate}}
	{}}

\usepackage[dvipsnames,svgnames,x11names,hyperref]{xcolor}
\PassOptionsToPackage{hyphens}{url}\usepackage{hyperref}
\hypersetup{
	colorlinks=true,
	citecolor=Mahogany,
	linkcolor=Mahogany,
	urlcolor=Mahogany,
	filecolor=Mahogany,
	bookmarksdepth=3
}

\usepackage[T1]{fontenc}
\usepackage{lmodern}
\usepackage{microtype}
\usepackage[mathscr]{euscript}

\usepackage{multicol}
\usepackage{amssymb}
\usepackage{amsmath}
\usepackage{amsthm}
\usepackage{mathtools}
\calclayout

\usepackage{tikz}
\usepackage{tikz-cd}
\usepackage{xparse}
\usepackage{etoolbox}
\usepackage{aliascnt}
\usepackage {hyperref}
\hypersetup{
	colorlinks=true,
	linkcolor=blue,
	filecolor=blue,
	citecolor = blue,
	urlcolor=blue,
	pdfencoding=unicode,
}
\usepackage[capitalize,nameinlink,noabbrev]{cleveref}

\AtBeginDocument{
	\def\MR#1{}
}

\usepackage{enumitem}
\setenumerate[1]{label=(\roman*),}
\setitemize[1]{label=\raisebox{0.25ex}{\boldmath$\cdot$}}

\usepackage[textsize = footnotesize]{todonotes}
\usepackage{comment}

\newcounter{environmentcounteralphabetic}

\numberwithin{environmentcounter}{section}

\newaliascnt{definitioncounteralias}{environmentcounter}

\crefname{definitioncounteralias}{Definition}{Definitions}
\newaliascnt{remarkcounteralias}{environmentcounter}

\crefname{remarkcounteralias}{Remark}{Remarks}
\newaliascnt{examplecounteralias}{environmentcounter}

\crefname{examplecounteralias}{Example}{Examples}
\newaliascnt{constructioncounteralias}{environmentcounter}

\crefname{constructioncounteralias}{Construction}{Constructions}
\newaliascnt{lemmacounteralias}{environmentcounter}

\crefname{lemmacounteralias}{Lemma}{Lemmas}
\newaliascnt{propositioncounteralias}{environmentcounter}

\crefname{propositioncounteralias}{Proposition}{Propositions}
\newaliascnt{corollarycounteralias}{environmentcounter}

\crefname{corollarycounteralias}{Corollary}{Corollaries}
\newaliascnt{theoremcounteralias}{environmentcounter}

\crefname{theoremcounteralias}{Theorem}{Theorems}
\newaliascnt{questioncounteralias}{environmentcounter}

\crefname{questioncounteralias}{Question}{Questions}
\newaliascnt{conjecturecounteralias}{environmentcounter}

\crefname{conjecturecounteralias}{Conjecture}{Conjectures}

\newaliascnt{theoremalphabeticcounteralias}{environmentcounteralphabetic}

\crefname{theoremalphabeticcounteralias}{Theorem}{Theorems}
\newaliascnt{corollaryalphabeticcounteralias}{environmentcounteralphabetic}

\crefname{corollaryalphabeticcounteralias}{Corollary}{Corollaries}

\theoremstyle{definition}
\newtheorem{definition}[definitioncounteralias]{Definition}

\theoremstyle{plain}
\newtheorem{lemma}[lemmacounteralias]{Lemma}
\newtheorem{proposition}[propositioncounteralias]{Proposition}
\newtheorem{corollary}[corollarycounteralias]{Corollary}
\newtheorem{question}[questioncounteralias]{Question}

\newtheorem{theorem}[theoremcounteralias]{Theorem}
\newtheorem{theoremalphabetic}[theoremalphabeticcounteralias]{Theorem}
\newtheorem{corollaryalphabetic}[corollaryalphabeticcounteralias]{Corollary}

\theoremstyle{remark}
\newtheorem{remark}[remarkcounteralias]{Remark}
\newtheorem{example}[examplecounteralias]{Example}

\makeatletter
\def\namedlabel#1#2{\begingroup
	#2
	\def\@currentlabel{#2}
	\phantomsection\label{#1}\endgroup
}
\makeatother

\newcommand{\bR}{\mathbb{R}}

\DeclareMathOperator{\PT}{PT}
\DeclareMathOperator{\RR}{R}
\DeclareMathOperator{\ST}{ST}
\DeclareMathOperator{\CT}{CT}
\DeclareMathOperator{\Pt}{Pt}

\DeclareMathOperator{\Ls}{Ls}
\DeclareMathOperator{\Tpl}{Tpl}
\DeclareMathOperator{\T}{T}
\newcommand{\cL}{\mathcal{L}}
\newcommand{\spa}{\on{span}}
\newcommand{\sma}{\wedge}

\newcommand{\on}[1]{\operatorname{#1}}

\newcommand{\colim}{\on{colim}}
\newcommand{\hocolim}{\on{hocolim}}

\newcommand{\apt}{\mathrm{apt}}

\newcommand{\category}[1]{\mathcal{#1}}

\newcommand{\aut}{\on{Aut}}

\newcommand{\Z}{\mathbb{Z}}

\newcommand{\R}{\mathbb{R}}

\begin{document}
	\begin{abstract}
		The Solomon--Tits theorem says that the poset of proper non-trivial subspaces of a finite-dimensional vector space has realisation equivalent to a wedge of spheres. In this paper we prove a variant of this result for collections of geodesic subspaces of Euclidean, hyperbolic, or spherical geometry, assuming the collection is generated either by points or by hyperplanes. In the third paper of this series of papers, we will combine this with the homological stability theorems from the first paper to compute the homology of groups of scissors automorphisms in these geometries.
	\end{abstract}

	\maketitle

	\vspace{-.5cm} \tableofcontents

	\vspace{-1cm} \section{Introduction}
	
	For a field $k$, the \emph{Tits building} $\T(k^n)$ is the geometric realisation of the poset of non-zero proper linear subspaces of $k^n$, ordered by inclusion. The Solomon--Tits theorem says that this space is homotopy equivalent to a wedge of $(n-2)$-spheres and its reduced homology in degree $(n-2)$ is a free abelian group generated by ``apartment'' classes (see \cite[Theorem 2]{Solomon} and \cite[Chapter IV.5, Theorem 2]{Brown}).
	
	In this paper we prove generalisations of the Solomon--Tits theorem for \emph{Tits complexes} $\T^{\cL}(X^n)$ that are constructed using a collection $\cL$ of complete geodesic subspaces in Euclidean, hyperbolic, or spherical geometry $X^n \in \{E^n, H^n, S^n\}$. For $k$ a subfield of $\R$, the Tits building is an example: there is an isomorphism $\T^{\cL}(S^{n-1}) \cong \T(k^n)$ when we let $\cL$ be the collection of geodesic subspaces of $S^{n-1}$ that arise as intersections $S^{n-1} \cap V$ of the sphere with a non-zero proper linear subspace $V \subseteq \R^n$ admitting a basis in $k^n \subseteq \R^n$. Our interest in such geometric variants of the Solomon--Tits theorem originates in the study of scissors congruence---this is the second paper in a series on this topic \cite{KLMMS-1,KLMMS-3}---but before explaining this in \cref{sec:connection} we will state our results precisely.
	
	\subsection{Polytopal Tits complexes}
	
	Consider the following three \emph{geometries} $X^n \subseteq \mathbb{R}^{n+1}$:
	\begin{enumerate}
		\item (Euclidean) $E^n = \{(x_0,\ldots,x_n) \mid x_0 = 1\}$,
		\item (hyperbolic) $H^n = \{(x_0,\ldots,x_n) \mid -x_0^2+\sum_{i=1}^n x_i^2 = -1 \text{ and }x_0>0\}$,
		\item (spherical) $S^n = \{(x_0,\ldots,x_n) \mid \sum_{i=0}^n x_i^2 = 1\}$.
	\end{enumerate}
	In any one of them, a \emph{geometric subspace} of $X^n$ is the intersection of $X^n$ with a linear subspace in $\R^{n+1}$. This is the same as a connected complete geodesic subspace of $X^n$, except in the case of a 0-dimensional subspace of $S^n$ where it gives a complete geodesic subspace with two components, namely a point and its antipodal point. Each of these geometries also has an associated \emph{isometry group} $I(X) \leq \mathrm{GL}_{n+1}(\mathbb{R})$:
	\begin{enumerate}
		\item $I(E^n) \cong \R^n \rtimes O(n)$ consists of linear maps preserving the form $\sum_{i=1}^n x_i^2$ and the function $x_0$.
		\item $I(H^n) = O(1,n)^+$ consists of linear maps preserving the form $-x_0^2+\sum_{i=1}^n x_i^2$ and the sign of $x_0$.
		\item $I(S^n) = O(n+1)$ consists of linear maps preserving the form $\sum_{i=0}^n x_i^2$.
	\end{enumerate}
	In the first case there is also a \emph{affine transformation group} $A(E^n) \coloneq \smash{(\R_{>0}^*)^n \ltimes I(E^n)}$ where we allow scalings in each of the coordinate directions, isomorphic to $\R^n \rtimes \mathrm{GL}_n(\mathbb{R})$.

	Let $X^n$ be one these geometries and $\cL$ be a collection of non-empty geometric subspaces $\varnothing \subsetneq U \subseteq X^n$ which includes $X^n$ itself. The \emph{$\cL$-Tits complex} is then defined to be geometric realisation of the poset of proper subspaces in $\cL$ under inclusion:
	\[\T^{\cL}(X^n) \coloneq |\cL \setminus \{X^n\}|. \]
	We also consider the following variants, which appear in the study of scissors congruence \cite{malkiewich2022,KLMMS-3}:

	\begin{definition}\label{df:st-and-pt}
	Fix a geometry $X^n$ and a collection $\cL$ of non-empty geometric subspaces of $X^n$ which includes $X^n$ itself.
		\begin{enumerate}
			\item The \emph{suspended Tits complex} $\ST^{\cL}(X^n)$ is defined to be
			\[ \ST^{\cL}(X^n) \coloneq \frac{ \underset{U \in \cL}\hocolim\, * }{ \underset{U \in \cL \setminus \{X^n\}}\hocolim\, * }.\]
			\item The \emph{polytopal Tits complex} $\PT^{\cL}(X^n)$ is defined to be
			\[ \PT^{\cL}(X^n) \coloneq \frac{ \underset{U \in \cL}\hocolim\, U }{ \underset{U \in \cL \setminus \{X^n\}}\hocolim\, U }.\]
		\end{enumerate}
	\end{definition}
	As $X^n$ is terminal in $\cL$, $\hocolim_{U \in \cL} \ast$ is contractible and hence $\ST^{\cL}(X^n)$ is equivalent to the unreduced suspension of $\smash{\T^{\cL}(X^n)}$. In the Euclidean and hyperbolic cases there is an equivalence $\smash{\ST^{\cL}(X^n) \simeq \PT^{\cL}(X^n)}$ since each geometric subspace $U \in \cL$ is contractible, but in the spherical case these complexes are in general not equivalent.

	\subsection{Solomon--Tits theorems and polytope groups}
	Our goal is to give conditions under which the complexes of \cref{df:st-and-pt} are equivalent to wedges of $n$-spheres, and then give a presentation of the $n$th reduced homology group---the analogue of the Steinberg module---equivariantly for the action of subgroups $G_\cL \leq I(X^n)$ (or $G_\cL \leq A(E^n)$) preserving $\cL$ set-wise. We shall consider two types of conditions on $\cL$; to state them, we let $\cL^k \subseteq \cL$ denote the subset of geometric subspaces of dimension $k$, so that $\cL^0 \subseteq \cL$ is the collection of points in $\cL$ (or, by convention, antipodal pairs of points in the spherical case), and $\cL^{n-1} \subseteq \cL$ is the collection of hyperplanes in $\cL$.

	\begin{definition}\label{def:generated_by}\,
		\begin{enumerate}
			\item We say $\cL$ is \emph{generated by points} if it is equal to the set of all possible spans of finite collections of points in $\cL^0$, other than $\varnothing$. (This implies that $\cL^0$ contains $(n+1)$ points in general position, since $X^n \in \cL$.)
			\item We say $\cL$ is \emph{generated by hyperplanes} if it is equal to the set of all possible intersections of finitely many hyperplanes in $\cL^{n-1}$, other than $\varnothing$. (This need not imply that $\cL^{n-1}$ has $(n+1)$ hyperplanes in general position; $X^n \in \cL$ can be obtained as the empty intersection.)
		\end{enumerate}
		Here the span of a set $S$ is the smallest geometric subspace containing $S$. In other words, it is the linear span in $\R^{n+1}$, intersected with $X^n$ to give a geometric subspace of $X^n$.
	\end{definition}
	
	\begin{example}If $k \subseteq \R$ is a subfield, we can take $\cL$ to be the collection of subspaces of $E^n$, $H^n$, or $S^n$ defined by linear equations in $\R^{n+1}$ with coefficients in $k$. Then $\cL$ is generated by hyperplanes and generated by points. If $k \subseteq \R$ only a subring, such as $\Z$, then $\cL$ as defined above is still generated by hyperplanes, but in general is not generated by points.\end{example}

	The easier case is when $\cL$ is generated by points, and then $\ST^\cL(X^n)$ is equivalent to a wedge of $n$-spheres and its $n$th reduced homology group admits a description analogous to Lee--Szczarba's for the Steinberg module in terms of simplices \cite[Theorem 3.1]{LeeSzczarba}.

	\begin{theoremalphabetic}\label{thm:generated-by-points} If $\cL$ is generated by points, then $\ST^\cL(X^n)$ is equivalent to a wedge of $n$-spheres and there is an isomorphism of $\Z[G_\cL]$-modules 
		\[\Ls^\cL(X^n) \overset{\cong}\longrightarrow \widetilde{H}_n(\ST^\cL(X^n)),\]
	whose domain is the restricted Lee--Szczarba group of \cref{def:lee-szczarba-group}.\end{theoremalphabetic}

	The harder case, and the one most relevant for scissors congruence, is when $\cL$ is generated by hyperplanes. In this case, under a mild admissibility condition that depends on the geometry (see \cref{admissible-en,admissible-hn,admissible-sn}), $\PT^\cL(X^n)$ is equivalent to a wedge of $n$-spheres and its $n$th reduced homology group admits a description in terms of convex $\cL$-polytopes.

	\begin{theoremalphabetic}\label{thm:generated-by-hyperplanes} If $\cL$ is generated by hyperplanes and is admissible, then $\PT^\cL(X^n)$ is equivalent to a wedge of $n$-spheres and there is an isomorphism of $\Z[G_\cL]$-modules 
		\[\Pt^\cL(X^n) \overset{\cong}\longrightarrow \widetilde{H}_n(\PT^\cL(X^n)),\]
	whose domain is the restricted polytope group of \cref{def:polytope-group}.
	\end{theoremalphabetic}
	
	\begin{remark}Though \cref{thm:generated-by-points} and \cref{thm:generated-by-hyperplanes} concern the different complexes $\ST^\cL(X^n)$ and $\PT^\cL(X^n)$, the complexes are equivalent in the Euclidean and hyperbolic case, and so we get two results about the same complex, holding under different assumptions. Moreover, if $\cL$ is generated by both points and hyperplanes then the polytope and Lee--Szczarba groups are isomorphic, $\Pt^\cL(X^n) \cong \Ls^\cL(X^n)$ (\cref{lem:gen-points-pt-to-ls}). When $\cL$ is only generated by hyperplanes, this is not necessarily the case---the Lee--Szczarba group gives the wrong answer because it only ``knows about simplices,'' and often $\cL$-polytopes can not be decomposed into $\cL$-simplices. For instance, this is true in the ``rectangular'' examples discussed in \cite{KLMMS-3}.
	\end{remark}
	
	In the spherical case, we prove that $\Ls^\cL(S^n)$ is the quotient of $\Pt^\cL(S^n)$ by suspensions of polytopes when $\cL$ is generated by hyperplanes and by points (\cref{lem:gen-points-pt-to-ls}). We also get a result for $\ST^\cL(S^n)$ when $\cL$ is generated by hyperplanes and admissible (\cref{admissible-sn}). In that case, we replace $\cL$ by its set of non-empty orthogonal complements $\cL^\perp$ (with $S^n \in \cL^\perp$ by convention). The admissibility condition says that the intersection of all the hyperplanes in $\cL$ is empty, and therefore the span of all the points in $\cL^\perp$ is all of $S^n$, so that $\cL^\perp$ is generated by points. This leads to a Lee--Szczarba presentation of the $n$th reduced homology group of $\ST^\cL(S^n)$ using tuples of hyperplanes in $\cL$ rather than tuples of points.

	\begin{corollaryalphabetic}\label{thm:generated-by-hyperplanes-spherical}
	If $X^n = S^n$ and $\cL$ is generated by hyperplanes and admissible, then $\ST^\cL(S^n)$ is equivalent to a wedge of $n$-spheres and there is an isomorphism of $\Z[G_\cL]$-modules 
	\[ \Ls^{\cL^\perp}(S^n) \overset{\cong} \longrightarrow \widetilde{H}_n(\ST^{\cL^\perp}\hspace{-.1cm}(S^n)) \cong \widetilde{H}_n(\ST^\cL(S^n)).\]
	\end{corollaryalphabetic}
	
	\subsection{Connection to scissors congruence and scissors automorphism groups}	\label{sec:connection}
	Our interest in geometric variants of the Solomon--Tits theorem---\cref{thm:generated-by-points} and \cref{thm:generated-by-hyperplanes}---originates in the subject of scissors congruence, which considers polytopes in a geometry $X^n$ up to moves that cut them into finitely many pieces and then rearrange the pieces. For a fixed polytope $P \subseteq X^n$, these cut-and-paste transformations define a group $\aut(P)$, the \emph{scissors automorphism group} of $P$. Examples include the groups of invertible piecewise isometries considered previously in the dynamics literature, see e.g.\ \cite{ag_pwi_pentagon_example,bp_pwi_classification,car_pwi_construction}. 
	
	We introduced these groups in the prequel \cite[Definition 1.1 and Definition 2.18]{KLMMS-1}, and explained they are to scissors congruence K-theory $K(\category{X}^n)$ of Zakharevich \cite{zakharevich2012, zakharevich2014} as the general linear groups are to the classical algebraic K-theory $K(R)$ of a ring \cite{quillen73, weibel2013}. From this perspective, the present work provides the geometric input required to generalize the computational tools that the third-named author developed for $K(\category{X}^n)$ \cite{malkiewich2022} to the setting of $\cL$-scissors congruence $K(\category{X}^{\cL}_G)$---the theory obtained by restricting to $\cL$-polytopes and allowing cut-and-paste operations in a group $G$ preserving $\cL$ set-wise. 
	
	More precisely, in the final paper of this series \cite{KLMMS-3} we establish the following generalization of the technical core of \cite{malkiewich2022}: in the setting of \cref{thm:generated-by-hyperplanes} (and \cref{thm:generated-by-hyperplanes-spherical}), $K(\category{X}^{\cL}_G)$ is the Thom spectrum of a bundle on $\smash{\PT^\cL(X^n)}$,
	\[K(\category{X}^{\cL}_G) \simeq \smash{\PT^\cL(X^n)}^{-TX^n}
	_{hG}.\]
	This description of $K(\category{X}^{\cL}_G)$ allows us to access the groups $K_i(\category{X}^{\cL}_G)$ via the spectrum homology $H_i(K(\category{X}^{\cL}_G)) = H_i(G;\smash{\Pt^\cL(X^n)^t})$. The assembler $\category{X}^{\cL}_G$ will frequently fit into the axiomatic framework of \cite[Section 3.1]{KLMMS-1}, so that for non-empty $\cL$-polytopes $P$, the homological stability result of \cite[Corollary 4.5]{KLMMS-1} shows that there is an acyclic map
	\[ B\aut_{\category{X}^{\cL}_G}(P) \longrightarrow \Omega^\infty_{[P]} K(\category{X}^{\cL}_G),\]
	where the subscript $[P]$ denotes the path component of $P$. Thus we can use scissors congruence K-theory to calculate the group homology of the scissors automorphism group $\smash{\aut_{\category{X}^{\cL}_G}(P)}$. This is used to determine the (co)homology of many families of scissors automorphism groups, recovering and extending results by Szymik--Wahl \cite{szymikwahl2019}, Li \cite{li2022} and Tanner \cite{tanner2023}.

	\subsection*{Acknowledgments}

	AK acknowledges the support of the Natural Sciences and Engineering Research Council of Canada (NSERC) [funding reference number 512156 and 512250]. AK is supported by an Alfred P.~Sloan Research Fellowship. CM was partially supported by the National Science Foundation (NSF) grants DMS-2005524, DMS-2052923, and DMS-2506430, a Simons Fellowship, and a Simons Travel Support for Mathematicians grant. JM was partially supported National Science Foundation (NSF) grants DMS-2202943 and DMS-2504473 as well as a Simons Foundation Travel Support for Mathematicians grant. RJS was supported by NSERC Discovery Grant A4000 as well as by the German Research Foundation through SFB 1442 -- 427320536, Geometry: Deformations and Rigidity, and EXC 2044 -- 390685587, Mathematics M\"unster: Dynamics--Geometry--Structure.

	The authors would also like to thank the organizers and participants of the Summer school on Scissors Congruence, Algebraic K-theory, and Trace Methods at IU Bloomington in June 2023 for conversations that inspired them to begin working on this paper.

	\section{Generation by points} In this section we prove \cref{thm:generated-by-points}. This says that $\ST^\cL(X^n)$ is equivalent to a wedge of $n$-spheres and that $\smash{\Ls^\cL(X^n) \cong \widetilde{H}_n(\ST^\cL(X^n))}$, if $\cL$ is \emph{generated by points}---that is, $\cL$ is equal to the set of spans of finite collections of points in $\cL^0$, other than $\varnothing$. We emphasise that, if $\cL$ generated by points, our assumption that $X^n \in \cL$ is thus equivalent to assuming that $\cL^0$ contains $(n+1)$ points in general position. We begin with this case because the proof strategy is well-known, see e.g.~\cite[Theorem 3.5]{dupont_book}, \cite[Corollary A.7 and Theorem 1.20]{cz}, \cite[Lemma 2.2]{calegari2022}, and \cite[Theorem 2.10]{malkiewich2022}.

	\subsection{The homotopy type of \texorpdfstring{$\ST^{\cL}(X^n)$}{ST\unichar{"005E}\unichar{"1D4DB}(X\unichar{"005E}n)}}  For the remainder of this section we fix a non-empty $\cL$ which is generated by points. We begin by analysing how the space $\ST^{\cL}(X^n)$ decomposes as a contractible space modulo a union of contractible spaces, one for each proper non-empty subspace of $X^n$. Recall the definition
		\[ \ST^{\cL}(X^n) \coloneq \frac{ \underset{U \in \cL}\hocolim\, * }{ \underset{U \in \cL \setminus \{X^n\}}\hocolim\, * }\]	
	where we take the homotopy colimit in unbased spaces. Concretely, we fix the model of this space obtained from the Bousfield--Kan formula \cite[Chapter XII]{BousfieldKan}: it is obtained by first forming the simplicial complex $\CT^{\cL}(X^n)$ with a vertex for each subspace $U \in \cL$ and a $p$-simplex for each flag $\varnothing \subsetneq U_0 \subsetneq U_1 \subsetneq \cdots \subsetneq U_p \subseteq X^n$. Then we collapse to a point the subcomplex $\smash{T^\cL(X^n) \subseteq \CT^\cL(X^n)}$ consisting of all simplices in which $U_p \neq X^n$.

	\begin{definition}\,
		\begin{itemize}
			\item For any set $S$, let $\Tpl(S)$ denote the full simplicial complex on $S$. It has a vertex for each element of $S$, and every finite set of vertices spans a simplex.
			\item Let $\Tpl(\cL^0)^{n-1} \subseteq \Tpl(\cL^0)$ be the subcomplex of those simplices that lie entirely a subspace $U \subsetneq X^n$ of dimension at most $n-1$.
			\item For each $U \in \cL \cup \{\varnothing\}$, let $\Tpl^U(\cL^0) \subseteq \Tpl(\cL^0)$ be the (full) subcomplex of those simplices that are contained in $U$. When $U = \varnothing$, this is an empty complex.
		\end{itemize}
	\end{definition}

	Note that when $U$ is not all of $X^n$, $\Tpl^U(\cL^0) \subset \Tpl(\cL^0)^{n-1}$, and in fact $\Tpl(\cL^0)^{n-1}$ is the union of all such subcomplexes $\Tpl^U(\cL^0)$. Moreover, for every $U,V \in \cL$, if we let $U \cap^\cL V$ be the maximum subspace in $\cL \cup \{\varnothing\}$ contained in $U \cap V$---which is well-defined because $\cL$ is generated by points---then we have
	\[\Tpl^U(\cL^0) \cap \Tpl^V(\cL^0) = \Tpl^{U \cap^\cL V}(\cL^0).\]
	We can thus write $\Tpl(\cL^0)^{n-1}$ as a homotopy colimit
	\[\underset{U \in \cL \setminus \{X^n\}}{\hocolim}\,\Tpl^U(\cL^0) \overset{\simeq}\longrightarrow \Tpl(\cL^0)^{n-1}.\]
	Now observe that $\Tpl(\cL^0)$ is contractible, since any element of $\cL^0$ is a cone point, as is $\Tpl^{U}(\cL^0)$ for each $U \in \cL$. Thus the quotient can be rewritten as
	\begin{align}\label{eqn:st_zigzag}
		\frac{\Tpl(\cL^0)}{\Tpl(\cL^0)^{n-1}}
		& \overset{\simeq}{\longleftarrow} \frac{ \Tpl(\cL^0)  }{ \underset{U \in \cL \setminus \{X^n\}}\hocolim\, \Tpl^U(\cL^0) }
		 \overset{\simeq}{\longrightarrow} \frac{ \underset{U \in \cL}\hocolim\, * }{ \underset{U \in \cL \setminus \{X^n\}}\hocolim\, * }
		\simeq \ST^\cL(X^n).
	\end{align}

	\begin{corollary}\label{steinberg_case_1}
		If $\cL$ is generated by points, $\ST^\cL(X^n)$ is equivalent to a wedge of $n$-spheres.
	\end{corollary}

	\begin{proof}
		This argument is similar to \cite[Theorem 2.10]{malkiewich2022}. Our fixed model for $\ST^\cL(X^n)$ only has cells of dimension $\leq n$, while the quotient $\Tpl(\cL^0) / \Tpl(\cL^0)^{n-1}$ only has (non-basepoint) cells of dimension $\geq n$. Since these are equivalent, they must be equivalent to a wedge of $n$-spheres.
	\end{proof}

	\subsection{The restricted Lee--Szczarba group and apartment-like maps}

	Next we define the restricted Lee--Szczarba group $\Ls^\cL(X^n)$ and prove that it maps to the reduced $n$th homology group of $\ST^\cL(X^n)$.

	\begin{definition}\label{def:lee-szczarba-group}
		The \emph{restricted Lee--Szczarba group} is the free abelian group on the $(n+1)$-tuples of elements of $\cL^0$, modulo the relation $\sum_{i=0}^{n+2} (-1)^i (y_0,\ldots,\hat{y_i},\ldots,y_{n+1}) = 0$ for each $(n+2)$-tuple $(y_0,\ldots,y_{n+1})$ and modulo all tuples lying in a subspace $U \subsetneq X^n$ of dimension at most $n-1$:
		\[ \Ls^\cL(X^n) = \frac{ \Z\langle (x_0,\ldots,x_n) \mid x_i \in \cL^0 \rangle }{ (x_0,\ldots,x_n) \in U \subsetneq X^n, \ \sum_{i=0}^{n+2} (-1)^i (y_0,\ldots,\hat{y_i},\ldots,y_{n+1}) }.\]
		We make it a $\Z[G_\cL]$-module by $g \cdot (x_0,\ldots,x_n) = (gx_0,\ldots,gx_n)$.
	\end{definition}

	\begin{remark}\label{rem:lee-szczarba-group}\,
		\begin{itemize}
			\item In \cref{def:lee-szczarba-group} repetitions are allowed, as is any ordering on the points. Alternatively, one could assume that repetitions are not allowed and the tuples all have to be ordered using a fixed total ordering on $\cL^0$; this gives an isomorphic group.
			\item In the sphere $S^n$, in \cref{def:lee-szczarba-group} we mean to take elements of $\cL^0$ which are really pairs of antipodal points $\{x,-x\}$.  Therefore, the Lee--Szczarba group $\Ls^\cL(S^n)$ can be identified with the group $\Ls^\cL(\R^{n+1})$ that is generated by $(n+1)$-tuples of lines in general position, modulo the simplicial relation for each $(n+2)$-tuple of lines. However, we \emph{could} take individual points $x \in \{x,-x\} \in \cL^0$ in the definition of $\Ls^\cL(S^n)$, and this gives an isomorphic group. In this latter case, when $-x_i$ denotes the antipode of $x_i \in S^n$, the two tuples
			\[\qquad \qquad (x_0,\ldots,x_{i-1},x_i,x_{i+1},\ldots,x_n), \quad \text{and} \quad (x_0,\ldots,x_{i-1},-x_i,x_{i+1},\ldots,x_n) \]
			represent the same element of the Lee--Szczarba group.
		\end{itemize}
	\end{remark}

	There is a map from the restricted Lee--Szczarba group to the homology of $\ST^\cL(X^n)$, given explicitly in terms of apartment classes. Given an $(n+1)$-tuple of points $(x_0,\ldots,x_n)$ in $X^n$ each permutation $\sigma$ of $\{0,\ldots,n\}$ gives a flag of spans
	\[ \spa(x_{\sigma(0)}) \subseteq \spa(x_{\sigma(0)},x_{\sigma(1)}) \subseteq \cdots \subseteq \spa(x_{\sigma(0)},x_{\sigma(1)},\ldots,x_{\sigma(n-1)}). \]
	(Recall that the span of a set in $X^n$ is the smallest geometric subspace containing that set---it is also the linear span in $\R^{n+1}$, intersected with $X^n$.) Since $\cL$ is generated by points, these spans are all elements of $\cL$. Together these $(n+1)!$ flags assemble to an \emph{apartment map}
	\begin{equation}\label{eqn:apt-map} S^{n-1} \cong |\mathrm{sd}\,\partial \Delta^n| \longrightarrow \T^\cL(X^n).\end{equation} 
	
	\vspace{1em}
	\centerline{
	\def\svgwidth{5in}
\begingroup%
  \makeatletter%
  \providecommand\color[2][]{%
    \errmessage{(Inkscape) Color is used for the text in Inkscape, but the package 'color.sty' is not loaded}%
    \renewcommand\color[2][]{}%
  }%
  \providecommand\transparent[1]{%
    \errmessage{(Inkscape) Transparency is used (non-zero) for the text in Inkscape, but the package 'transparent.sty' is not loaded}%
    \renewcommand\transparent[1]{}%
  }%
  \providecommand\rotatebox[2]{#2}%
  \newcommand*\fsize{\dimexpr\f@size pt\relax}%
  \newcommand*\lineheight[1]{\fontsize{\fsize}{#1\fsize}\selectfont}%
  \ifx\svgwidth\undefined%
    \setlength{\unitlength}{373.01231548bp}%
    \ifx\svgscale\undefined%
      \relax%
    \else%
      \setlength{\unitlength}{\unitlength * \real{\svgscale}}%
    \fi%
  \else%
    \setlength{\unitlength}{\svgwidth}%
  \fi%
  \global\let\svgwidth\undefined%
  \global\let\svgscale\undefined%
  \makeatother%
  \begin{picture}(1,0.28371266)%
    \lineheight{1}%
    \setlength\tabcolsep{0pt}%
    \put(0,0){\includegraphics[width=\unitlength,page=1]{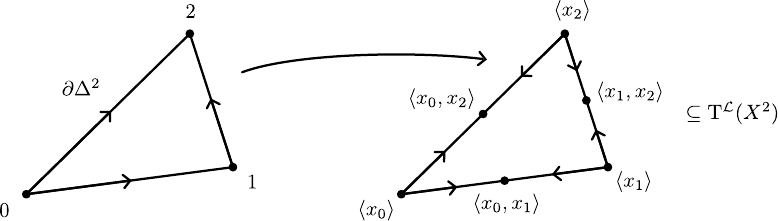}}%
  \end{picture}%
\endgroup%

	}
	\vspace{1em}

	\begin{definition}\label{apartment_st}
	 Coning off \eqref{eqn:apt-map} gives the \emph{suspended apartment map},
	\begin{equation}\label{eqn:susp-apt-map}
		S^n \cong \frac{\Delta^n}{\partial \Delta^n} \longrightarrow \frac{\CT^\cL(X^n)}{\T^\cL(X^n)} = \ST^\cL(X^n)
	\end{equation}
	and the image of the fundamental class of $S^n$ under it is the \emph{apartment class}
	\[ \apt(x_0,\ldots,x_n) \in \widetilde{H}_{n}(\ST^\cL(X^n)). \]
	\end{definition}
	
	\vspace{1em}
	\centerline{
	\def\svgwidth{5in}
\begingroup%
  \makeatletter%
  \providecommand\color[2][]{%
    \errmessage{(Inkscape) Color is used for the text in Inkscape, but the package 'color.sty' is not loaded}%
    \renewcommand\color[2][]{}%
  }%
  \providecommand\transparent[1]{%
    \errmessage{(Inkscape) Transparency is used (non-zero) for the text in Inkscape, but the package 'transparent.sty' is not loaded}%
    \renewcommand\transparent[1]{}%
  }%
  \providecommand\rotatebox[2]{#2}%
  \newcommand*\fsize{\dimexpr\f@size pt\relax}%
  \newcommand*\lineheight[1]{\fontsize{\fsize}{#1\fsize}\selectfont}%
  \ifx\svgwidth\undefined%
    \setlength{\unitlength}{377.71136205bp}%
    \ifx\svgscale\undefined%
      \relax%
    \else%
      \setlength{\unitlength}{\unitlength * \real{\svgscale}}%
    \fi%
  \else%
    \setlength{\unitlength}{\svgwidth}%
  \fi%
  \global\let\svgwidth\undefined%
  \global\let\svgscale\undefined%
  \makeatother%
  \begin{picture}(1,0.28017846)%
    \lineheight{1}%
    \setlength\tabcolsep{0pt}%
    \put(0,0){\includegraphics[width=\unitlength,page=1]{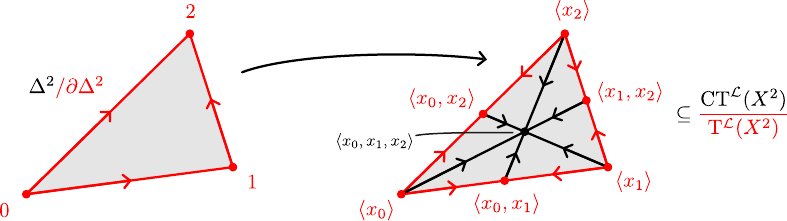}}%
  \end{picture}%
\endgroup%

	}
	\vspace{1em}

	\begin{lemma}\label{st-apt-well-defined}If $\cL$ is generated by points, there is a well-defined map of $\Z[G_\cL]$-modules
		\begin{align*}\Ls^\cL(X^n) &\longrightarrow \widetilde{H}_{n}(\ST^\cL(X^n)) \\
			(x_0,\ldots,x_n) &\longmapsto \apt(x_0,\ldots,x_n).\end{align*}
	\end{lemma}

	\begin{proof}
	We first prove that for an $(n+2)$-tuple of points $(y_0,\ldots,y_{n+1})$ the relation \[\sum_{i=0}^{n+1} (-1)^i \apt(y_0,\ldots,\widehat{y}_i,\ldots,y_{n+1}) = 0\] holds.
	To do so, the $(n+2)$-tuple induces a map $\Delta^{n+1}/\mathrm{sk}_{n-1}(\Delta^{n+1}) \to \CT^\cL(X^n)/\T^\cL(X^n)$, which exhibits the sum of its faces as a boundary, signs arising from the orientation induced by the standard orientation of $\Delta^{n+1}$ on its faces.

	We next prove that for $(x_0,\ldots,x_n)$ contained in a subspace $U \subsetneq X^n$ of dimension at most $n-1$ the relation $\apt(x_0,\ldots,x_n) = 0$ holds. To do so, we use \cref{steinberg_case_1} to prove that this map factors over the map $\ST^{\cL \cap U}(U) \to \ST^\cL(X^n)$, which is a map from a wedge of $(n-1)$-spheres to a wedge of $n$-spheres and is therefore null-homotopic.
	\end{proof}

	\subsection{Solomon--Tits when \texorpdfstring{$\cL$}{\unichar{"1D4DB}} is generated by points} To prove \cref{thm:generated-by-points} it remains to check that the map of \cref{st-apt-well-defined} is an isomorphism. To do this we will construct an isomorphism $\Ls^\cL(X^n) \cong \widetilde H_n(\Tpl(\cL^0) / \Tpl(\cL^0)^{n-1})$, compose it with the isomorphism induced on reduced homology by the equivalence from \eqref{eqn:st_zigzag}, and check that the resulting isomorphism $\smash{\Ls^\cL(X^n) \cong \widetilde{H}_n(\ST^\cL(X^n))}$ takes each element $(x_0,\ldots,x_n)$ to the corresponding apartment class.

	To set this up, observe that \eqref{eqn:st_zigzag} is a zig-zag of equivalences that all have a particular form: there is a contractible complex (the ``numerator'') quotiented by a non-contractible subcomplex (the ``denominator''). In particular, to calculate its effect on homology, it suffices to understand how apartment maps behave under the equivalences of ``denominators'':
	\begin{equation}\label{eqn:t_zigzag}
	\begin{tikzcd}
		\Tpl(\cL^0)^{n-1} & \ar[l,"\simeq" swap] \underset{U \in \cL \setminus \{X^n\}}\hocolim\, \Tpl^{U}(\cL^0) \ar[r,"\simeq"] & \underset{U \in \cL \setminus \{X^n\}}\hocolim\, * = \T^\cL(X^n).
	\end{tikzcd}
	\end{equation}
	
	Given $(x_0,\ldots,x_n)$, each face $F \subseteq \partial \Delta^n$ corresponds to a proper subset of the set $\{0,\ldots,n\}$ and hence of the set $\{x_0,\ldots,x_n\}$ of $0$-dimensional subspaces, whose affine span in $X^n$ we call $\spa F$. A map from $\partial \Delta^n$ into any of the three spaces in \eqref{eqn:t_zigzag} is \emph{apartment-like} if it sends each face $F \subseteq \partial \Delta^n$ into the contractible subcomplex that only involves subspaces of $\spa F$. The following definition makes this precise.

	\begin{definition}\label{st-apt-like} Fix an $(n+1)$-tuple $(x_0,\ldots,x_n)$ in $\cL^0$. 
	\begin{itemize}
		\item A map $\partial \Delta^n \to \Tpl(\cL^0)^{n-1}$ is \emph{apartment-like} if each face $F \subseteq \partial \Delta^n$ is sent into the contractible subcomplex $\Tpl^{\spa F}(\cL^0)$.
		\item A map 
		\[\partial \Delta^n \longrightarrow \left(\underset{U \in \cL \setminus \{X^n\}}\hocolim\, \Tpl^{U}(\cL^0)\right)\]
		 is \emph{apartment-like} if each face $F \subseteq \partial \Delta^n$ is sent into the contractible subcomplex $\hocolim_{U \in \cL, \ U \subseteq \spa F} \Tpl^{U}(\cL^0)$.
		\item A map $\partial \Delta^n \to \T^\cL(X^n)$ is \emph{apartment-like} if each face $F \subseteq \partial \Delta^n$ is sent into the contractible subcomplex $\CT^\cL(\spa F)$ of all subspaces contained in $\spa F$, including $\spa F$ itself (hence the $\CT$ instead of the $\T$).
	\end{itemize}
	\end{definition}

	\begin{lemma}\label{apt-like-preserved}
		For any fixed $(n+1)$-tuple $(x_0,\ldots,x_n)$ in $\cL^0$ and each of the three spaces in \eqref{eqn:t_zigzag}, the space of apartment-like maps is weakly contractible. Moreover,  composition with the maps in \eqref{eqn:t_zigzag} preserves apartment-like maps.
	\end{lemma}

	\begin{proof}
		We prove that the space of apartment-like maps is weakly contractible using the method of \cite[Lemma 4.12]{malkiewich2022}. To do the argument in a uniform way, let $T$ be one of the complexes in \eqref{eqn:t_zigzag}, and for each face $F \subseteq \partial \Delta^n$ let $CT_{\spa F}$ be the contractible subcomplex of $T$ where $F$ is allowed to go, implicitly depending on $(x_0,\ldots,x_n)$. Note that when $F \subseteq F'$, $CT_{\spa F} \subseteq CT_{\spa F'}$, in other words passing from $F$ to $CT_{\spa F}$ respects inclusion.

		To prove weak contractibility, given an $(m-1)$-sphere of apartment-like maps
		\[ S^{m-1} \times \partial \Delta^n \to T,  \]
		we extend it to an $m$-disc of apartment-like maps $D^m \times \partial \Delta^n \to T$. It suffices to consider cells of the form $D^m \times F$ for $F \subseteq \partial \Delta^n$, to assume that the extension has been defined on $S^{m-1} \times F$ and $D^m \times \partial F$, and construct it to $D^m \times F$. Since the complexes $CT_{\spa F}$ respect inclusion, the given map lands in the complex $CT_{\spa F}$:
		\[ (S^{m-1} \times F) \cup_{(S^{m-1} \times \partial F)} (D^m \times \partial F) \to CT_{\spa F}. \]
		Since $CT_{\spa F}$ is contractible, there exists an extension of this to $D^m \times F$. Repeating for each cell of $D^m \times \partial \Delta^n \to T$ in this way proves that the extension exists.

		The fact that apartment-like maps are preserved under composition with the arrows in \eqref{eqn:t_zigzag} follows from the definition.
	\end{proof}

	\begin{corollary}\label{apt-like-preserved-2}
		Any suspension of an apartment-like map
		\[ \frac{\Delta^n}{\partial \Delta^n} \longrightarrow \frac{\Tpl(\cL^0)}{\Tpl(\cL^0)^{n-1}}, \]
		when composed with the equivalence of \eqref{eqn:st_zigzag} goes to a map that is homotopic to the suspended apartment map \eqref{eqn:susp-apt-map}
		\[ \frac{\Delta^n}{\partial \Delta^n} \longrightarrow \frac{\CT^\cL(X^n)}{\T^\cL(X^n)} = \ST^\cL(X^n). \]
	\end{corollary}

	\begin{theorem}[Solomon-Tits for $\cL$ generated by points]\label{steinberg_case}
		If $\cL$ is generated by points, $\ST^\cL(X^n)$ is equivalent to a wedge of $n$-spheres	and the apartment class map induces a well-defined isomorphism
		\[\apt \colon \Ls^\cL(X^n) \overset{\cong}\longrightarrow \widetilde{H}_n(\ST^\cL(X^n)).\]
	\end{theorem}

	\begin{proof}
	We already showed in \cref{steinberg_case_1} that $\ST^\cL(X^n)$ is a wedge of $n$-spheres, so we only have to prove that the apartment map is an isomorphism. Choose a total ordering on the vertices of $\Tpl(\cL^0)$ and form its simplicial chain complex. If we quotient out by all the simplices in $\Tpl(\cL^0)^{n-1}$, the result is a cellular chain complex for the quotient space $\Tpl(\cL^0) / \Tpl(\cL^0)^{n-1}$. We arrive at a formula for $\widetilde H_n(\Tpl(\cL^0) / \Tpl(\cL^0)^{n-1})$: a generator for each non-degenerate $(n+1)$-tuple of points, modulo the simplicial relation. In other words, we get the same presentation as $\Ls^\cL(X^n)$ (compare with \cref{rem:lee-szczarba-group}).

	This gives an isomorphism
	\[ \Ls^\cL(X^n) \cong \widetilde H_n(\Tpl(\cL^0) / \Tpl(\cL^0)^{n-1}), \]
	where each generator goes to the suspension of the canonical map $\partial \Delta^n \to \Tpl(\cL^0)^{n-1}$ that sends the vertices of $\Delta^n$ to the points $x_i \in X^n$. This canonical map is apartment-like, so by \cref{apt-like-preserved-2}, if we compose with the isomorphism
	\[ H_n(\Tpl(\cL^0) / \Tpl(\cL^0)^{n-1}) \cong H_n(\ST^\cL(X^n)) \]
	induced by the equivalence $\eqref{eqn:st_zigzag}$, we get the map $\Ls^\cL(X^n) \to \widetilde{H}_n(\ST^\cL(X^n))$ that sends each generator to its apartment class. We have now shown that this map is a composition of isomorphisms, and is therefore an isomorphism.
	\end{proof}

	This finishes the proof of \cref{thm:generated-by-points} and from it we can deduce \cref{thm:generated-by-hyperplanes-spherical}.
	
	\begin{proof}[Proof of \cref{thm:generated-by-hyperplanes-spherical}] Suppose $\cL$ is a collection of subspaces in the sphere $S^n$ whose intersection is empty, or equivalently subspaces of $\R^{n+1}$ whose intersection is the zero subspace. We consider the dual collection $\cL^\perp$ of orthogonal complements of the subspaces in $\cL$ (including $S^n$ itself and not including $\varnothing$). This sends spans to intersections and vice-versa, so $\cL$ is generated by hyperplanes (and has $\bigcap \cL = \varnothing$) if and only if $\cL^\perp$ is generated by points (which requires that its span is all of $S^n$).
	
	The duality map $U \mapsto U^\perp$ induces an isomorphism of Tits complexes $\T^\cL(S^n) \cong \T^{\cL^\perp}(S^n)$ that reverses the order of the poset. Hence if $\cL$ is generated by hyperplanes and is admissible, $\T^\cL(S^n)$ is a wedge of $(n-1)$-spheres and there is an isomorphism proving \cref{thm:generated-by-hyperplanes-spherical}
	\[ \Ls^{\cL^\perp}(S^n) \overset{\cong}\longrightarrow \widetilde{H}_n(\ST^{\cL^\perp}(X^n)) \cong \widetilde{H}_n(\ST^{\cL}(X^n)).\qedhere\]
	\end{proof}
	
	\begin{remark}Note that in the further case that $\cL$ is generated by both points and hyperplanes, the isomorphism of Tits complexes induces an isomorphism of groups
		\[ \Ls^{\cL}(S^n) \cong \Ls^{\cL^\perp}(S^n), \]
		and in the further case that $\cL^\perp = \cL$, this becomes an involution on the Lee--Szczarba group.\end{remark}
	
	\begin{remark}\label{st-n0}
		The arguments of this section work even when $n = 0$. In the case of $S^0$,  ``generated by points'' means that $\cL = \{S^0\}$. Thus $T^\cL(S^0) = \varnothing$, $\ST^\cL(S^0) \cong S^0$, and the apartment class defines an isomorphism $\Z \cong \Ls^\cL(S^0) \longrightarrow \widetilde{H}_0(\ST^\cL(S^0)) \cong \Z$.
	\end{remark}

	\section{Generation by hyperplanes and the Euclidean case}

	In the next three sections we prove \cref{thm:generated-by-hyperplanes}: it says that $\PT^\cL(X^n)$ is equivalent to a wedge of $n$-spheres and there is an isomorphism $\smash{\Pt^\cL(X^n) \cong \widetilde{H}_n(\PT^\cL(X^n))}$, if $\cL$ is \emph{generated by hyperplanes}---that is, $\cL$ is equal to the set of intersections of finite collections of hyperplanes in $\cL^{n-1}$, other than $\varnothing$---and \emph{admissible}. The notion of admissible depends on the geometry, and appears to be necessary.

	\subsection{The restricted polytope group} Our first objective is to define the $\cL$-polytope group in a geometry $X^n$, fixed for the remainder of this section.

	\begin{definition}\label{df:pt}
	Suppose $\cL$ is generated by hyperplanes.
		\begin{enumerate}
			\item A \emph{convex $\cL$-polytope} in $X^n$ is a subset $P \subseteq X^n$ obtained as an intersection of $X^n$ with finitely many half-spaces in $\bR^{n+1}$ defined by hyperplanes in $\cL^{n-1}$, such that $P$ is bounded and not contained in proper geometric subspace $U \subsetneq X^n$.
			\item An \emph{$\cL$-polytope} is a subset $P \subseteq X$ that is a finite union of convex $\cL$-polytopes.
		\end{enumerate}
	\end{definition}
	
	We refer to the condition in \cref{df:pt} (i) that $P$ is not contained in proper geometric subspace $U \subsetneq X^n$ as $P$ being \emph{$n$-dimensional}, as it is equivalent to its affine span being $X^n$. If $P$ is an $\cL$-polytope, a \emph{subdivision} of $P$ is a finite union of $\cL$-polytopes $\{P_i\}$ whose union is $P$, and whose interiors are disjoint. (This is sometimes called a \emph{weak subdivision} or an \emph{almost-disjoint cover}. The smaller polytopes $P_i$ are not required to intersect along common faces.) Without loss of generality (subdividing further), the polytopes $P_i$ are obtained by intersecting $P$ with finitely many hyperplanes in $\cL^{n-1}$, and taking the closures of all the resulting $n$-dimensional regions.

	\begin{definition}\label{def:polytope-group}
		The \emph{$\cL$-polytope group} $\Pt^\cL(X^n)$ is the free abelian group on the $\cL$-polytopes $P \subseteq X^n$, modulo the relation that whenever $P$ subdivides into finitely many $\cL$-polytopes $\{P_i\}$, we have $[P] = \sum_i [P_i]$:
		\[ \Pt^\cL(X^n) = \frac{ \Z\langle [P] \rangle }{ [P] = \sum_i [P_i] \textup{ if $P$ subdivides into $\{P_i\}$ } }. \]
		We make it a $\Z[G_\cL]$-module by $g \cdot [P] = [gP]$.
	\end{definition}

	\begin{remark}
	Every polytope can be subdivided into convex polytopes. As a result, we get the same group $\Pt^\cL(X^n)$ if we only take convex polytopes in the generators and the relations.
	\end{remark}

	\begin{remark}
		\label{rem:spherical-case-polytope-group}
		In the spherical case, we call a polytope that is convex in the sense of \cref{df:pt} \emph{weakly convex}. We call it \emph{strongly convex} if in addition it lies in an open hemisphere. We define $\Pt^\cL(S^n)$ using weakly convex polytopes because it gives a sensible answer even when $\cL$ has very few hyperplanes. For instance, if $\cL$ has no hyperplanes, then $\Pt^\cL(S^n) \cong \Z$, generated by $S^n$, which is a weakly convex polytope in $S^n$.
	\end{remark}
	
	We record the following fact from the theory of convex polytopes. A \emph{face} of a convex $\cL$-polytope $P$ is a subset $F$ obtained by intersecting with any hyperplane $U$ such that $P$ is contained in one of the half-spaces $U^-$ or $U^+$ defined by $U$. A \emph{facet} is an $(n-1)$-dimensional face $F$, defined by a unique hyperplane $U = \spa F$.
	\begin{proposition}\label{faces_to_L}
		Each convex $\cL$-polytope $P$ has a minimal collection of half-spaces that define it. The corresponding hyperplanes $U_1, \cdots, U_k$ are the spans of the facets of $P$. Furthermore, for every face $F$ of $P$, we have
		\[ \spa F = \bigcap_{F \subseteq U_i} U_i \quad \textup{and} \quad F = P \cap \spa F. \]
	\end{proposition}
	
	Therefore each convex $\cL$-polytope is defined by a minimal collection of hyperplanes in $\cL^{n-1}$, namely the spans of the facets of $P$. Furthermore, the span of every face is an element of $\cL$, obtained as an intersection of some of these hyperplanes.
	
	The proof of \cref{faces_to_L} follows from the standard theory in e.g.~\cite[Section 2.2]{ziegler}. Note that this theory works equally well for convex hyperbolic polytopes and strongly convex spherical polytopes, since we can identify those with convex Euclidean polytopes using the embeddings of all three geometries into $\R^{n+1}$. The case of weakly convex spherical polytopes then follows easily from the observation made in \cref{pt-suspended-0} below.
	
	\begin{remark}
		We warn the reader that intersecting an arbitrary subset of the $U_i$ together does not give the span of a face in general. Only if we take all of the hyperplanes that contain $F$ do we get the span of $F$.
	\end{remark}

	It can be useful to think of an element of $\Pt^\cL(X^n)$ as a measure on $X^n$. To make this precise, let $C(X^n,\Z)$ denote the abelian group of equivalence classes of functions $X^n \to \Z$, where two functions are equivalent if the set where they are not equal is contained in the union of finitely many hyperplanes in $X^n$.
	
	\begin{lemma}\label{pt_subgroup}
		Assigning a polytope its indicator function yields an injective homomorphism
		\begin{align*} \Pt^\cL(X^n) &\longrightarrow C(X^n,\Z) \\
			P &\longmapsto 1_P.\end{align*}
	\end{lemma}

	\begin{proof}The indicator function $1_P$ is 1 on the interior of $P$ and 0 on the exterior; it is not necessary to define the function the boundary $\partial P$, since this is contained in a finite union of hyperplanes. It is clear this gives a well-defined homomorphism $\Pt^\cL(X^n) \to C(X^n,\Z)$. To see that it also is injective, we suppose that the element $\sum_i n_i[P_i] \in \Pt^\cL(X^n)$ goes to the equivalence class of the zero function in $C(X^n,\Z)$. Take the finite collection of hyperplanes that define the polytopes $P_i$ and use the subdivision relation in $\Pt^\cL(X^n)$ to cut all of the polytopes along all of these hyperplanes. We get an equivalent expression of the form $\sum_i m_i[Q_i]$, in which the interiors of the $Q_i$ are all disjoint. Since this goes to the equivalence class of a function that is zero on the interior of each $Q_i$, we must have $m_i = 0$ for all $i$, hence our element of $\Pt^\cL(X^n)$ is zero.
	\end{proof}

	\subsection{Apartments and apartment-like maps for \texorpdfstring{$\Pt^\cL(X^n)$}{Pt\unichar{"005E}\unichar{"1D4DB}(X\unichar{"005E}n)}} In contrast to the previous section, we start our proof with the construction of a map of $\Z[G_\cL]$-modules $\Pt^\cL(X^n) \to \widetilde{H}_n(\PT^\cL(X^n))$. Recall that $\PT^{\cL}(X^n)$ is defined to be
	\[ \PT^{\cL}(X^n) \coloneq \frac{ \underset{U \in \cL}\hocolim\, U }{ \underset{U \in \cL \setminus \{X^n\}}\hocolim\, U },\]
	where we take the homotopy colimit in unbased spaces. Again we fix the model of this obtained from the Bousfield--Kan formula. The resulting construction differs from that of $\ST^{\cL}(X^n)$ in that we get a copy of the product $\Delta^p \times U_0$ for each flag of subspaces $\varnothing \subsetneq U_0 \subsetneq U_1 \subsetneq \cdots \subsetneq U_p \subseteq X^n$, joined along faces in the obvious way, and then we collapse to a point the subspace consisting of all simplices in which $U_p$ is not equal to $X^n$.
	
	In this section and the next we will handle the Euclidean and hyperbolic cases, where the evident map $\PT^\cL(X^n) \to \ST^\cL(X^n)$ is an equivalence because all $U$ are contractible. For simplicity we therefore work with $\ST^\cL(X^n)$. The maps we define to $\ST^\cL(X^n)$ are also interesting in the spherical case, even though they are not necessarily isomorphisms; see \cref{lem:map-pt-to-ls}. The following is a variant of \cref{st-apt-like}:

	\begin{definition}\label{st-apt-like-2}
		If $P$ is a convex $\cL$-polytope, a map $\partial P \to \T^\cL(X^n)$ is \emph{apartment-like} if for every face $F \subseteq \partial P$ of any dimension, the image of $F$ lands in the contractible subcomplex $\CT^\cL(\spa F)$ of subspaces contained in the affine span of $F$.
	\end{definition}

	\begin{lemma}\label{apt-like-preserved-3-pt}
		The space of apartment-like maps $\partial P \to \T^\cL(X^n)$ is weakly contractible.
	\end{lemma}

	\begin{proof}
		The proof is the same as in \cref{apt-like-preserved}, except that we are mapping out of $D^m \times \partial P$, and we think of $\partial P$ as a cell complex with one cell for each face of $P$. It may happen in the case of $X^n = S^n$ that the lowest-dimensional face has dimension greater than zero, but in that case it is a sphere $S^k \subseteq S^n$. If so, we choose some cell complex structure on this sphere, before proceeding with the extension argument.
	\end{proof}
	
	To get a consistent choice of signs in the next definition, we fix once and for all an orientation on $X^n$, which induces in turn an orientation on the sphere $P/\partial P$ for each convex polytope $P$.

	\begin{definition}\label{apartment_pt}
		For each convex polytope $P$, we define a \emph{suspended apartment map}
		\[S^n \cong \frac{P}{\partial P} \longrightarrow \frac{\CT^\cL(X^n)}{\T^\cL(X^n)} = \ST^\cL(X^n)\]
		by coning off any apartment-like map $\partial P \to \T^\cL(X^n)$. The image of the fundamental class of $P/\partial P$ under the induced map on homology is the \emph{apartment class} 
		\[ \apt(P) \in \widetilde{H}_n(\ST^\cL(X^n)). \]
	\end{definition}

	\begin{lemma}\label{pt-apt-well-defined}The apartment classes induce a well-defined homomorphism
		\[\apt \colon \Pt^\cL(X^n) \longrightarrow \widetilde{H}_{n}(\ST^\cL(X^n)).\]
	\end{lemma}

	\begin{proof}
		Let $P$ be an $\cL$-polytope which is subdivided into finitely many $\cL$-polytopes $\{P_i\}$. Assume that there exists a map $\cup_i \partial P_i \to \T^\cL(X^n)$ that is apartment-like for each $P_i$. Then, this map is automatically apartment-like when restricted to $\partial P$ and hence gives the commutativity of the diagram
		\[ \begin{tikzcd}[column sep=4em]
			P/\partial P \ar[d,"\textup{collapse}"'] \ar[r,"\textup{apt}"] & \ST^{\cL}(X^n). \\
			\bigvee_i P_i/\partial P_i \ar[ru,"\textup{apt}"']
		\end{tikzcd} \]
	By our convention for orientations, if $P$ is subdivided into $P_i$ then the collapse map $P/\partial P \to \bigvee_i P_i/\partial P_i$ is degree one on each generator. It follows that the generator $[P]$ in the polytope group is sent to a homology class in $\widetilde H_n(\ST^{\cL}(X^n))$ that is the sum of the images of the elements $[P_i]$, as required.

	It remains to prove the existence of a map $\cup_i \partial P_i \to \T^\cL(X^n)$ that is apartment-like for each $P_i$. To this end, we take all the hyperplanes defining the $P_i$ and all their intersections to construct a polytopal complex, which contains a subcomplex that identifies with $\cup_i \partial P_i$. Then we define a map from this complex to $\T^\cL(X^n)$, inductively up the dimensions of its faces, sending every face $F$ to the subcomplex $\CT^\cL(\spa F)$. Restriction to $\cup_i \partial P_i$ then yields the desired map $\cup_i \partial P_i \to \T^\cL(X^n)$.
	\end{proof}

	\subsection{Preliminaries for the Euclidean case}
	Let $\cL$ be a set of non-empty affine subspaces of the Euclidean geometry $E^n$ that is generated by hyperplanes, i.e.~it consists of those non-empty subspaces obtained by taking intersections of finitely many hyperplanes in $\cL^{n-1}$. By convention the empty intersection $E^n$ is also in $\cL$. We will prove a Solomon--Tits theorem for $\T^\cL(E^n)$ under the following assumption:

	\begin{definition}\label{admissible-en}
		For $n \geq 1$, a collection $\cL$ of subspaces of $E^n$ that is generated by hyperplanes is \emph{admissible} if  $\cL^0$ is non-empty (that is, there exists a point in $E^n$ that is an intersection of hyperplanes from $\cL^{n-1}$).
	\end{definition}
	
	This is equivalent to asking for the set of hyperplanes $\cL^{n-1}$ to have rank equal to $n$, where the rank is the dimension of the linear space spanned by the normal vectors to the hyperplanes (\cite[Lecture 1]{stanley}).
	
	\begin{remark}
		We will prove under this hypothesis that $\T^\cL(E^n)$ is a wedge of $(n-1)$-spheres. However, suppose $\cL$ is generated by hyperplanes but $\cL^0 = \varnothing$. Then if $k$ is the lowest dimension of subspace in $\cL$, every other subspace in $\cL$ is parallel to some fixed $k$-dimensional one, and therefore our result proves that $\T^\cL(E^n)$ is a wedge of $(n-k-1)$-spheres.
	\end{remark}
	
	For the inductive argument, we will need to consider various collections of subspaces formed out of a given $\cL$ and a hyperplane $U$.

	\begin{definition}\label{cup_and_cap_def}
		Suppose $\cL$ is generated by hyperplanes and $U \subset E^n$ is a single hyperplane. We define:
		\begin{enumerate}
			\item $\cL \Cup U$ as the collection of non-empty subspaces in $E^n$ that are intersections of finitely many hyperplanes in $\cL^{n-1} \cup \{U\}$,
			\item $\cL \uplus U = (\cL \Cup U)\setminus \{U\}$ as the collection of subspaces that are intersections of finitely many hyperplanes in $\cL^{n-1} \cup \{U\}$, but not the subspace $U$ itself, and
			\item $\cL \cap U$ as the collection of subspaces in $\cL \Cup U$ that are contained in $U$. In other words, the set of non-empty subspaces of the form $L \cap U$ for $L \in \cL$.
		\end{enumerate}
		Note that $\cL \cap U$ always contains $U = E^n \cap U$.
	\end{definition}

	\vspace{1em}
	\centerline{
	\def\svgwidth{6.5in}
\begingroup%
  \makeatletter%
  \providecommand\color[2][]{%
    \errmessage{(Inkscape) Color is used for the text in Inkscape, but the package 'color.sty' is not loaded}%
    \renewcommand\color[2][]{}%
  }%
  \providecommand\transparent[1]{%
    \errmessage{(Inkscape) Transparency is used (non-zero) for the text in Inkscape, but the package 'transparent.sty' is not loaded}%
    \renewcommand\transparent[1]{}%
  }%
  \providecommand\rotatebox[2]{#2}%
  \newcommand*\fsize{\dimexpr\f@size pt\relax}%
  \newcommand*\lineheight[1]{\fontsize{\fsize}{#1\fsize}\selectfont}%
  \ifx\svgwidth\undefined%
    \setlength{\unitlength}{465.56599581bp}%
    \ifx\svgscale\undefined%
      \relax%
    \else%
      \setlength{\unitlength}{\unitlength * \real{\svgscale}}%
    \fi%
  \else%
    \setlength{\unitlength}{\svgwidth}%
  \fi%
  \global\let\svgwidth\undefined%
  \global\let\svgscale\undefined%
  \makeatother%
  \begin{picture}(1,0.24967258)%
    \lineheight{1}%
    \setlength\tabcolsep{0pt}%
    \put(0,0){\includegraphics[width=\unitlength,page=1]{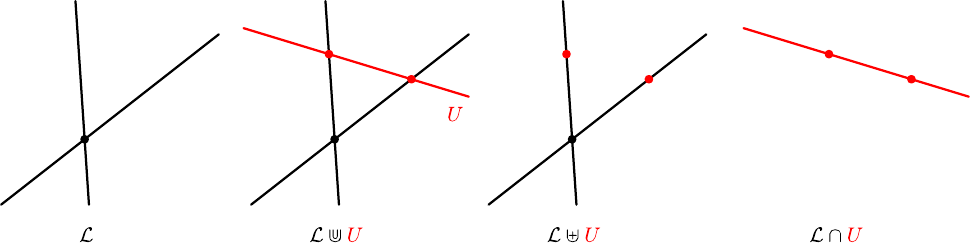}}%
  \end{picture}%
\endgroup%

	}
	\vspace{1em}

	\begin{lemma}\label{admissible_inherited}
		If $\cL$ is admissible in $E^n$ and $n \geq 2$, then for any hyperplane $U \in \cL^{n-1}$ the collection $\cL \cap U$ is also admissible in $U$. Hence, every subspace in $\cL$ contains a point in $\cL^0$.
	\end{lemma}

	\begin{proof}
	As $\cL \cap U$ is generated by those hyperplanes in $U$ of the form $L \cap U$ for $L \in \cL^{n-1}$, it remains to show that it contains a 0-dimensional subspace. Let $L_1 \cap \cdots \cap L_n = \{x\}$ be a collection of hyperplanes in $\cL$ intersecting at one point. As remarked after \cref{admissible-en}, this means that their normal vectors from a basis $\{w_i\}$ of $\R^n$. The normal vector $w$ of $U$ is thus in the linear hull of $\{w_i\}$, and it follows from Steinitz exchange lemma that there exists some $i$ such that $\{w_j\}_{j \neq i} \cup \{w\}$ is a basis of $\R^n$. Using the remark after \cref{admissible-en} again, it follows that $\cap_{j \neq i} L_i \cap U$ is a point.
	\end{proof}

	Observe that if $\cL$ is admissible in $E^n$, then $\cL \Cup U$ also admissible in $E^n$ and if $n \geq 2$ then $\cL \cap U$ is admissible in $U$ (by \cref{admissible_inherited}, as $U$ is an element of $\cL \Cup U$).
	
	\begin{lemma}\label{new_generators}
		If $\cL$ is admissible, $n \geq 1$, and $U \notin \cL$ is a hyperplane, then any $(n-1)$-dimensional $(\cL \cap U)$-polytope $P \subseteq U$ has a subdivision in which each piece $P_i$ is a facet of a convex $n$-dimensional $(\cL \Cup U)$-polytope $Q_i \subseteq E^n$.
	\end{lemma}
	
	\vspace{1em}
	\centerline{
	\def\svgwidth{5.2in}
\begingroup%
  \makeatletter%
  \providecommand\color[2][]{%
    \errmessage{(Inkscape) Color is used for the text in Inkscape, but the package 'color.sty' is not loaded}%
    \renewcommand\color[2][]{}%
  }%
  \providecommand\transparent[1]{%
    \errmessage{(Inkscape) Transparency is used (non-zero) for the text in Inkscape, but the package 'transparent.sty' is not loaded}%
    \renewcommand\transparent[1]{}%
  }%
  \providecommand\rotatebox[2]{#2}%
  \newcommand*\fsize{\dimexpr\f@size pt\relax}%
  \newcommand*\lineheight[1]{\fontsize{\fsize}{#1\fsize}\selectfont}%
  \ifx\svgwidth\undefined%
    \setlength{\unitlength}{335.70091985bp}%
    \ifx\svgscale\undefined%
      \relax%
    \else%
      \setlength{\unitlength}{\unitlength * \real{\svgscale}}%
    \fi%
  \else%
    \setlength{\unitlength}{\svgwidth}%
  \fi%
  \global\let\svgwidth\undefined%
  \global\let\svgscale\undefined%
  \makeatother%
  \begin{picture}(1,0.29421624)%
    \lineheight{1}%
    \setlength\tabcolsep{0pt}%
    \put(0,0){\includegraphics[width=\unitlength,page=1]{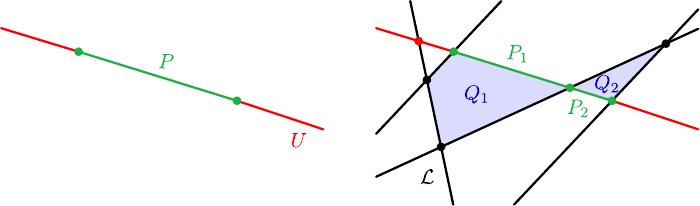}}%
  \end{picture}%
\endgroup%

	}
	\vspace{1em}
	
	\begin{proof}
		Consider the finite collection of hyperplanes in $\cL$ bounding half-spaces that intersect with $U$ to form $P$. If there is no way to form a 0-dimensional space by intersecting some of these hyperplanes, then since $\cL$ is admissible, we can add finitely many additional hyperplanes from $\cL$ to the collection until some subset of them has a 0-dimensional intersection. These additional hyperplanes might further subdivide $P$.
		By considering each piece in this subdivision separately, we may now assume that the $(\cL \cap U)$-polytope $P$ is cut out by intersecting $U$ with half-spaces $L_1^+, \ldots, L_m^+$ bounded by a collection of hyperplanes $L_1, \ldots, L_m \in \cL^{n-1}$ such that some subset of the $L_i$ intersects to form a 0-dimensional subspace. So
		\[ P = L_1^+ \cap \cdots \cap L_m^+ \cap U. \]
		Note that not all of the hyperplanes $L_i$ define facets of $P$; in the figure below, there are three such lines, but only two of them define facets of $P$.

	\vspace{1em}
	\centerline{
	\def\svgwidth{5.8in}
\begingroup%
  \makeatletter%
  \providecommand\color[2][]{%
    \errmessage{(Inkscape) Color is used for the text in Inkscape, but the package 'color.sty' is not loaded}%
    \renewcommand\color[2][]{}%
  }%
  \providecommand\transparent[1]{%
    \errmessage{(Inkscape) Transparency is used (non-zero) for the text in Inkscape, but the package 'transparent.sty' is not loaded}%
    \renewcommand\transparent[1]{}%
  }%
  \providecommand\rotatebox[2]{#2}%
  \newcommand*\fsize{\dimexpr\f@size pt\relax}%
  \newcommand*\lineheight[1]{\fontsize{\fsize}{#1\fsize}\selectfont}%
  \ifx\svgwidth\undefined%
    \setlength{\unitlength}{386.32300499bp}%
    \ifx\svgscale\undefined%
      \relax%
    \else%
      \setlength{\unitlength}{\unitlength * \real{\svgscale}}%
    \fi%
  \else%
    \setlength{\unitlength}{\svgwidth}%
  \fi%
  \global\let\svgwidth\undefined%
  \global\let\svgscale\undefined%
  \makeatother%
  \begin{picture}(1,0.25566342)%
    \lineheight{1}%
    \setlength\tabcolsep{0pt}%
    \put(0,0){\includegraphics[width=\unitlength,page=1]{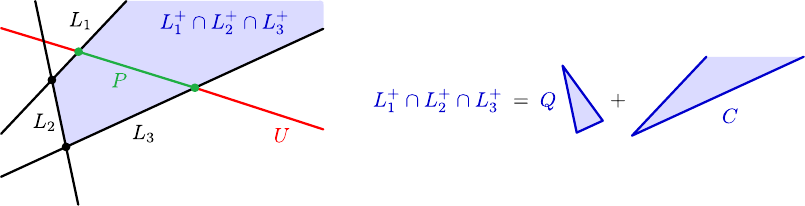}}%
  \end{picture}%
\endgroup%

	}
	\vspace{1em}
		
		The intersection $L_1^+ \cap \cdots \cap L_m^+$ is a convex but not necessarily bounded polyhedron. By the finite basis theorem \cite[Theorem 1.2]{ziegler}, this is expressed as a Minkowski sum of a polytope $Q$ and a convex cone $C$. That is, $C$ is the set of nonnegative linear combinations of some finite set of vectors $\{v_1,\ldots,v_k\}$, and we have
		\begin{equation}\label{eq:cone}
			L_1^+ \cap \cdots \cap L_m^+ = Q + C = \{ \ q + c_1v_1 + \cdots + c_kv_k \ \mid \ q \in Q, \ c_i \geq 0 \ \}.
		\end{equation}
		Since the intersection of the set \eqref{eq:cone} with $U$ is bounded, none of the non-zero vectors in $C$ can be parallel to $U$. Since the hyperplanes $L_1, \ldots, L_m$ have a subset with 0-dimensional intersection, up to a linear change of coordinates, $C$ is a cone in $\R^n$ contained in the set $\{ (x_1,\ldots,x_n) \mid \forall i, \ x_i \geq 0 \}$. Therefore, regardless of the coordinate system we use, $C$ cannot contain a vector and its negative, and more generally it cannot contain a set of the form $\{ v, -\lambda v\}$, with $v \neq 0$ and $\lambda > 0$.

		It follows that $C$ must only contain vectors pointing to one side of $U$, or more precisely there is a single vector $w \in \bR^n$ normal to $U$ such that $v \cdot w > 0$ for every $v \in C$. Otherwise, taking $v,v' \in C$ such that $v \cdot w > 0 > v' \cdot w$, either $v' = -\lambda v$ which is a contradiction, or else the positive linear combination
		\[ -(v' \cdot w)v + (v \cdot w)v' \in C \]
		is non-zero, but it is perpendicular to $w$ and therefore parallel to $U$, a contradiction.

		Letting $U^-$ be the half-space on the opposite side of $U$ from $w$, the intersection $(Q+C) \cap U^-$ does not contain any infinite ray, since the direction of the ray would have a positive dot product with $w$, and would therefore eventually reach the other half-space $U^+$. Therefore $(Q+C) \cap U^-$ is bounded and hence a polytope \cite[Def 0.1]{ziegler}. It is an $(\cL \Cup U)$-polytope, because $(Q+C) \cap U^- = L_1^+ \cap \cdots \cap L_m^+ \cap U^-$.
		
		It remains to see that $(Q+C) \cap U^-$ is $n$-dimensional and has $P$ as a facet. Since $P$ is $(n-1)$-dimensional, it has a point $x$ in its interior. The point $x$ must also be in the interior of $L_1^+ \cap \cdots \cap L_m^+$, since otherwise it would lie in an intersection $L_i \cap U$, and therefore be in the boundary of $P$. Therefore some small open ball about $x$ is contained in the interior of $L_1^+ \cap \cdots \cap L_m^+$, and intersecting this with $U^-$ gives an open set contained in $(Q+C) \cap U^-$. Therefore $(Q+C) \cap U^-$ is $n$-dimensional. It has $P$ as a facet by construction, because $(Q + C) \cap U = L_1^+ \cap \cdots \cap L_m^+ \cap U = P$.
	\end{proof}

	\begin{lemma}\label{pt_quotient}
		If $\cL$ is admissible, $n \geq 1$, and $U \notin \cL$ is a hyperplane, then there is an exact sequence of abelian groups
		\[ 0 \longrightarrow \Pt^{\cL}(E^n) \longrightarrow \Pt^{\cL \Cup U}(E^n) \longrightarrow \Pt^{\cL \cap U}(U) \longrightarrow 0. \]
	\end{lemma}
	
	Note that this fails when $U \in \cL$ because then the left map is an isomorphism but the third group is non-trivial.

	\begin{proof}
		The right map $\Pt^{\cL}(E^n) \to \Pt^{\cL \Cup U}(E^n)$ is injective as it fits in a commutative square
		\[\begin{tikzcd}\Pt^\cL(E^n) \rar \dar[hook] & \Pt^{\cL \Cup U}(E^n) \dar[hook] \\[-5pt]
			C(X^n,\Z) \rar[equal] & C(X^n,\Z)
		\end{tikzcd}\]
		where the vertical maps are injective by \cref{pt_subgroup}.

		Fixing a choice of normal direction to $U$ we get half-spaces $U^+,U^- \subseteq E^n$. We construct a map
		\begin{equation}\label{pt_quotient_identification-en}
			\Pt^{\cL \Cup U}(E^n) \big / \Pt^{\cL}(E^n) \longrightarrow \Pt^{\cL \cap U}(U)
		\end{equation}
		by taking each convex $(\cL \Cup U)$-polytope $Q$ to its facet (if any) that lies along $U$, with positive sign if $Q \subseteq U^+$ and negative sign if $Q \subseteq U^-$. Note that if neither of these containments hold then $Q$ does not have a facet along $U$, and the image of $Q$ along \eqref{pt_quotient_identification-en} is zero. This map is well-defined because it sends a convex $\cL$-polytope to zero, and subdividing along a hyperplane either does not change the result (if the hyperplane does not intersect $U$), subdivides the result (the hyperplane intersects $U$ transversally), or adds new polytopes that cancel each other out (the hyperplane is $U$ itself).
		
	\vspace{1em}
	\centerline{
	\def\svgwidth{5.3in}
\begingroup%
  \makeatletter%
  \providecommand\color[2][]{%
    \errmessage{(Inkscape) Color is used for the text in Inkscape, but the package 'color.sty' is not loaded}%
    \renewcommand\color[2][]{}%
  }%
  \providecommand\transparent[1]{%
    \errmessage{(Inkscape) Transparency is used (non-zero) for the text in Inkscape, but the package 'transparent.sty' is not loaded}%
    \renewcommand\transparent[1]{}%
  }%
  \providecommand\rotatebox[2]{#2}%
  \newcommand*\fsize{\dimexpr\f@size pt\relax}%
  \newcommand*\lineheight[1]{\fontsize{\fsize}{#1\fsize}\selectfont}%
  \ifx\svgwidth\undefined%
    \setlength{\unitlength}{359.62324161bp}%
    \ifx\svgscale\undefined%
      \relax%
    \else%
      \setlength{\unitlength}{\unitlength * \real{\svgscale}}%
    \fi%
  \else%
    \setlength{\unitlength}{\svgwidth}%
  \fi%
  \global\let\svgwidth\undefined%
  \global\let\svgscale\undefined%
  \makeatother%
  \begin{picture}(1,0.24066846)%
    \lineheight{1}%
    \setlength\tabcolsep{0pt}%
    \put(0,0){\includegraphics[width=\unitlength,page=1]{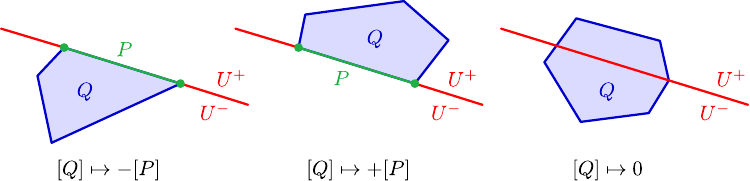}}%
  \end{picture}%
\endgroup%

	}
	\vspace{1em}
	
		We construct an inverse to \eqref{pt_quotient_identification-en} by taking each convex $(\cL \cap U)$-polytope $P \subseteq U$, using \cref{new_generators} to subdivide it into pieces $P_i$ that are facets of convex $(\cL \Cup U)$-polytopes $Q_i$ in $E^n$, then taking the polytopes $\pm [Q_i]$, again with a positive sign when $Q_i \subseteq U^+$ and a negative sign when $Q_i \subseteq U^-$. It is clear that this is an inverse if it is well-defined.

		To show that our proposed inverse is well-defined, we first observe that any two subdivisions of $P$ in which the pieces are faces of convex $(\cL \Cup U)$-polytopes in $E^n$, have a common refinement in which each piece is the face of a convex $(\cL \Cup U)$-polytope in $E^n$. So without loss of generality, we have one convex polytope $P \subseteq U$ that is the face of two convex polytopes $Q, Q' \subseteq E^n$ and must show that these two polytopes give the same element of the quotient $\Pt^{\cL \Cup U}(E^n) \big / \Pt^{\cL}(E^n)$.
		
	\vspace{1em}
	\centerline{
	\def\svgwidth{1.8in}
\begingroup%
  \makeatletter%
  \providecommand\color[2][]{%
    \errmessage{(Inkscape) Color is used for the text in Inkscape, but the package 'color.sty' is not loaded}%
    \renewcommand\color[2][]{}%
  }%
  \providecommand\transparent[1]{%
    \errmessage{(Inkscape) Transparency is used (non-zero) for the text in Inkscape, but the package 'transparent.sty' is not loaded}%
    \renewcommand\transparent[1]{}%
  }%
  \providecommand\rotatebox[2]{#2}%
  \newcommand*\fsize{\dimexpr\f@size pt\relax}%
  \newcommand*\lineheight[1]{\fontsize{\fsize}{#1\fsize}\selectfont}%
  \ifx\svgwidth\undefined%
    \setlength{\unitlength}{120.94785273bp}%
    \ifx\svgscale\undefined%
      \relax%
    \else%
      \setlength{\unitlength}{\unitlength * \real{\svgscale}}%
    \fi%
  \else%
    \setlength{\unitlength}{\svgwidth}%
  \fi%
  \global\let\svgwidth\undefined%
  \global\let\svgscale\undefined%
  \makeatother%
  \begin{picture}(1,0.63641793)%
    \lineheight{1}%
    \setlength\tabcolsep{0pt}%
    \put(0,0){\includegraphics[width=\unitlength,page=1]{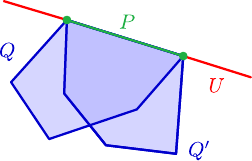}}%
  \end{picture}%
\endgroup%

	}
	\vspace{1em}
	
		The first case is that $Q$ and $Q'$ are on the same side of $U$, and we need to show that the difference $[Q']-[Q]$ lies in the subgroup $\Pt^{\cL}(E^n)$. Break up the union $Q \cup Q'$ into three pieces: the intersection $Q \cap Q'$ and the differences $Q \setminus Q'$ and $Q' \setminus Q$. The intersection is a convex polytope that also has $P$ as its facet along $U$. The differences are polytopes (finite unions of convex polytopes) that cannot have any facets along $U$, since their intersection with $U$ is contained in $\partial P$. Therefore the differences define elements of $\Pt^{\cL}(E^n)$, and we have
		\[ [Q']-[Q] = [Q' \setminus Q] - [Q \setminus Q'] \in \Pt^{\cL}(E^n). \]

	\vspace{1em}
	\centerline{
	\def\svgwidth{1.8in}
\begingroup%
  \makeatletter%
  \providecommand\color[2][]{%
    \errmessage{(Inkscape) Color is used for the text in Inkscape, but the package 'color.sty' is not loaded}%
    \renewcommand\color[2][]{}%
  }%
  \providecommand\transparent[1]{%
    \errmessage{(Inkscape) Transparency is used (non-zero) for the text in Inkscape, but the package 'transparent.sty' is not loaded}%
    \renewcommand\transparent[1]{}%
  }%
  \providecommand\rotatebox[2]{#2}%
  \newcommand*\fsize{\dimexpr\f@size pt\relax}%
  \newcommand*\lineheight[1]{\fontsize{\fsize}{#1\fsize}\selectfont}%
  \ifx\svgwidth\undefined%
    \setlength{\unitlength}{120.94785273bp}%
    \ifx\svgscale\undefined%
      \relax%
    \else%
      \setlength{\unitlength}{\unitlength * \real{\svgscale}}%
    \fi%
  \else%
    \setlength{\unitlength}{\svgwidth}%
  \fi%
  \global\let\svgwidth\undefined%
  \global\let\svgscale\undefined%
  \makeatother%
  \begin{picture}(1,0.78727894)%
    \lineheight{1}%
    \setlength\tabcolsep{0pt}%
    \put(0,0){\includegraphics[width=\unitlength,page=1]{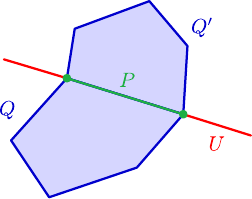}}%
  \end{picture}%
\endgroup%

	}
	\vspace{1em}
	
		The second case is that $Q$ and $Q'$ are on opposite sides of $U$, and we need to show that the sum $[Q']+[Q]$ lies in the subgroup $\Pt^{\cL}(E^n)$. The union $Q \cup Q'$ is a polytope, not necessarily convex, and we have to show that it can be defined using only the hyperplanes in $\cL$, not using $U$. If we take all the half-spaces defining $Q$ other than $U$, they define a polyhedron that contains $P$ and is not necessarily bounded. However, the hyperplanes defining $Q'$ other than $U$ must bound this region, for if not then there is some ray in $E^n$ that starts in $Q$, crosses $U$, and then continues without ever touching the sides of $Q'$, which is impossible because $Q'$ is bounded. Therefore, if we cut up $E^n$ along all the hyperplanes defining $Q$ and $Q'$ other than $U$, we get finitely many bounded regions, whose union contains $Q \cup Q'$. Each of these bounded regions is either inside the union $Q \cup Q'$ or outside, as otherwise there would be a straight-line path from the interior of $Q \cup Q'$ to the exterior that never crosses the hyperplanes defining $Q$ and $Q'$ (other than $U$), and this is impossible. In summary, $Q \cup Q'$ is a union of convex $\cL$-polytopes, so that $[Q]+[Q'] \in \Pt^{\cL}(E^n)$ as required.

	\vspace{1em}
	\centerline{
	\def\svgwidth{1.8in}
\begingroup%
  \makeatletter%
  \providecommand\color[2][]{%
    \errmessage{(Inkscape) Color is used for the text in Inkscape, but the package 'color.sty' is not loaded}%
    \renewcommand\color[2][]{}%
  }%
  \providecommand\transparent[1]{%
    \errmessage{(Inkscape) Transparency is used (non-zero) for the text in Inkscape, but the package 'transparent.sty' is not loaded}%
    \renewcommand\transparent[1]{}%
  }%
  \providecommand\rotatebox[2]{#2}%
  \newcommand*\fsize{\dimexpr\f@size pt\relax}%
  \newcommand*\lineheight[1]{\fontsize{\fsize}{#1\fsize}\selectfont}%
  \ifx\svgwidth\undefined%
    \setlength{\unitlength}{119.62316696bp}%
    \ifx\svgscale\undefined%
      \relax%
    \else%
      \setlength{\unitlength}{\unitlength * \real{\svgscale}}%
    \fi%
  \else%
    \setlength{\unitlength}{\svgwidth}%
  \fi%
  \global\let\svgwidth\undefined%
  \global\let\svgscale\undefined%
  \makeatother%
  \begin{picture}(1,0.7959971)%
    \lineheight{1}%
    \setlength\tabcolsep{0pt}%
    \put(0,0){\includegraphics[width=\unitlength,page=1]{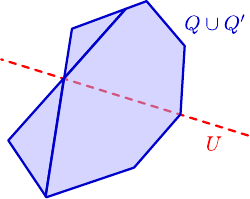}}%
  \end{picture}%
\endgroup%

	}
	\vspace{1em}
	
	\end{proof}

	\subsection{Proof in the Euclidean case}
	
	\begin{theorem}[Solomon-Tits for $\cL$-polytopes in $E^n$]\label{main-en}
		For all $n \geq 1$, if $\cL$ is admissible then $\T^\cL(E^n)$ is equivalent to a wedge of $(n-1)$-spheres and the apartment classes of \cref{apartment_pt} induce an isomorphism
		\[\apt \colon \Pt^\cL(E^n) \overset{\cong}\longrightarrow \widetilde H_{n}(\ST^\cL(E^n)). \]
	\end{theorem}

	\begin{proof}The proof is by induction on $n$.
	In the initial case $n = 1$, $\cL$ is admissible so long as $\cL^0 = \cL \setminus \{E^1\}$ is non-empty. In this case, we recognise $\ST^\cL(E^1)$ as the unreduced suspension of the set $\cL^0$, so it is equivalent a wedge of circles. Moreover, the map
		\[\apt \colon \Pt^\cL(E^1) \longrightarrow \widetilde{H}_1(\ST^\cL(E^1)) \cong \ker(\Z\langle \cL^0 \rangle \overset{\epsilon}\longrightarrow \Z)\]
	sends the convex polytope given by the interval bounded by $p<p'$ to the class $[p']-[p]$. It is easy to check that this map is an isomorphism.

	\smallskip
	
	We now assume $n \geq 2$ and that the theorem has been proven for the previous value of $n$. We shall apply Zorn's lemma to the partially ordered set $\mathcal{P}_\cL$ of admissible subsets $\cL' \subseteq \cL$ for which the conclusion of the theorem is true, ordered by inclusion. There are three main steps:
	\begin{itemize}
		\item proving that $\mathcal{P}_\cL$ is non-empty,
		\item adding one hyperplane $U$ to $\cL' \in \mathcal{P}_\cL$, and proving that $\cL' \Cup U \in \mathcal{P}_\cL$, and
		\item taking a colimit of collections $\cup_{\lambda} \cL_\lambda$ with $\cL_\lambda \in \mathcal{P}_\cL$ and proving that $\cup_{\lambda} \cL_\lambda \in \mathcal{P}_\cL$.
	\end{itemize}
	
	For the first step, we prove $\mathcal{P}_\cL$ is non-empty. Since $\cL$ is admissible, we can find $n$ hyperplanes in $\cL^{n-1}$ whose intersection is a point $x \in E^n$. Let $\cL_\circ \subseteq \cL$ consist of these hyperplanes and their intersections. Then $\T^{\cL_\circ}(E^n)$ is contractible as $x$ is an initial object, while $\smash{\Pt^{\cL_\circ}(E^n)} = 0$ as every region cut out by these hyperplanes is unbounded, so $\cL_\circ \in \mathcal{P}_\cL$.
	
	\medskip

	For the second key inductive step, we prove that for $\cL' \in \mathcal{P}_{\cL}$ and a hyperplane $U \in \cL \backslash \cL'$ we have $\cL' \Cup U \in \mathcal{P}_{\cL}$. Recall from \cref{cup_and_cap_def} that $\cL' \Cup U$ is the collection obtained from $\cL'$ by adding in $U$ and all of its intersections with subspaces in $\cL'$, whereas $\cL' \uplus U$ means that we add the intersections of $U$ with subspaces of $\cL'$ but not $U$ itself. We can thus factor the inclusion $\smash{\T^{\cL'}(E^n) \to \T^{\cL' \Cup U}(E^n)}$ as the top row of the following diagram.
	\begin{equation}\label{eqn:fact-pushout} \begin{tikzcd} \T^{\cL'}(E^n) \rar{\simeq} &  \T^{\cL' \uplus U}(E^n) \rar & \T^{\cL' \Cup U}(E^n) \\[-5pt]
		& \T^{\cL' \cap U}(U) \rar \uar & \CT^{\cL' \cap U}(U) \uar \end{tikzcd}\end{equation}	
	In this commutative diagram the square is a pushout, as the downward link of $U$ in $\smash{\T^{\cL' \Cup U}(E^n)}$ comes from the subposet of those subspaces that are properly contained in $U$, which is by definition $\smash{\T^{\cL' \cap U}(U)}$. Moreover, the top-left map of \eqref{eqn:fact-pushout} is an equivalence because it admits a homotopy inverse, given by the map of posets
	\[
		\big(V \in \cL' \uplus U \big) \longmapsto \left(\bigcap_{V \subseteq V' \in \cL'} V' \in \cL' \right).\]
	This is well-defined, as any element of $\cL' \uplus U$ is obtained by intersecting some elements of $\cL'$ possibly with $U$. It is an order-increasing map, $V \subseteq f(V)$, that is the identity on $\smash{\T^{\cL'}(E^n)}$, so induces a deformation retract upon realisation.

	\vspace{1em}
	\centerline{
	\def\svgwidth{4.5in}
\begingroup%
  \makeatletter%
  \providecommand\color[2][]{%
    \errmessage{(Inkscape) Color is used for the text in Inkscape, but the package 'color.sty' is not loaded}%
    \renewcommand\color[2][]{}%
  }%
  \providecommand\transparent[1]{%
    \errmessage{(Inkscape) Transparency is used (non-zero) for the text in Inkscape, but the package 'transparent.sty' is not loaded}%
    \renewcommand\transparent[1]{}%
  }%
  \providecommand\rotatebox[2]{#2}%
  \newcommand*\fsize{\dimexpr\f@size pt\relax}%
  \newcommand*\lineheight[1]{\fontsize{\fsize}{#1\fsize}\selectfont}%
  \ifx\svgwidth\undefined%
    \setlength{\unitlength}{306.0804212bp}%
    \ifx\svgscale\undefined%
      \relax%
    \else%
      \setlength{\unitlength}{\unitlength * \real{\svgscale}}%
    \fi%
  \else%
    \setlength{\unitlength}{\svgwidth}%
  \fi%
  \global\let\svgwidth\undefined%
  \global\let\svgscale\undefined%
  \makeatother%
  \begin{picture}(1,0.49351584)%
    \lineheight{1}%
    \setlength\tabcolsep{0pt}%
    \put(0,0){\includegraphics[width=\unitlength,page=1]{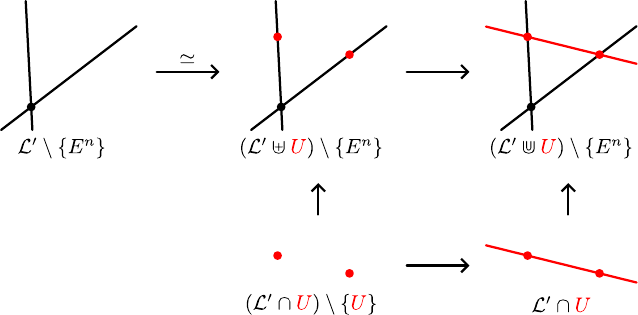}}%
  \end{picture}%
\endgroup%

	}
	\vspace{1em}

	By \cref{admissible_inherited}, the collection of hyperplanes $\smash{\cL' \cap U}$ is admissible in $U$, so $\smash{\T^{\cL' \cap U}(U)}$ is a wedge of $(n-2)$-spheres by the inductive hypothesis. As we assumed that $\smash{\T^{\cL'}(E^n)}$ is a wedge of $(n-1)$-spheres, we get a homotopy pushout square of the form
	\begin{equation*}\label{eqn:fact-pushout-2} \begin{tikzcd} \bigvee S^{n-1} \rar & \T^{\cL' \Cup U}(E^n) \\[-5pt]
		\bigvee S^{n-2} \rar \uar & {*}\uar \end{tikzcd}
	\end{equation*}
	in which the left-hand vertical map must be nullhomotopic, and so $\T^{\cL' \Cup U}(E^n)$ is equivalent to a wedge of $(n-1)$-spheres, completing the proof of the first half of the theorem for $\cL' \Cup U$.

	The second half of the theorem requires us to prove that the apartment map is an isomorphism. The above homotopy pushout square provides a short exact sequence
	\[0 \longrightarrow \widetilde{H}_n(\ST^{\cL'}(E^n)) \longrightarrow \widetilde{H}_n(\ST^{\cL' \Cup U}(E^n)) \longrightarrow \widetilde{H}_{n-1}(\ST^{\cL' \cap U}(U)) \longrightarrow 0.\]
	As the apartment map is natural with respect to inclusions of collections of subspaces $\cL$, using \cref{pt_quotient} we can extend this to a map of short exact sequences
	\begin{equation}\label{st_en_map_of_les}
	\begin{tikzcd}[column sep = 1.5em] 0 \rar & \Pt^{\cL'}(E^n) \rar \dar{\apt}[swap]{\cong} & \Pt^{\cL' \Cup U}(E^n) \dar{\apt} \rar & \Pt^{\cL' \cap U}(U) \dar[dashed] \rar & 0 \\[-5pt]
		0 \rar & \widetilde{H}_n(\ST^{\cL'}(E^n)) \rar &  \widetilde{H}_n(\ST^{\cL' \Cup U}(E^n)) \rar & \widetilde{H}_{n-1}(\ST^{\cL' \cap U}(U)) \rar & 0.\end{tikzcd}
	\end{equation}
	We claim that the dashed map agrees with the apartment map inside the geometry $U$, if we give $U$ the ``induced'' orientation from $U^+$, which inherits an orientation from $E^n$. To see this, we identify the reduced homology of $\ST$ with the homology of the pair $(\CT,\T)$, so that apartment maps are given by maps of pairs $(Q,\partial Q) \to (\CT,\T)$. Under this identification, the bottom-right surjective homomorphism in \eqref{st_en_map_of_les} can be rewritten as
	
	\begin{equation}\label{snake_map_expanded}
		\begin{tikzcd}
		 H_n(\CT^{\cL' \Cup U}(E^n),\T^{\cL' \Cup U}(E^n)) \rar{\cong} \arrow[phantom, ""{coordinate, name=Z}]{d} & \widetilde{H}_{n-1}(\T^{\cL' \Cup U}(E^n))
			\arrow[
			rounded corners,
			to path={
				-- ([xshift=2ex]\tikztostart.east)
				|- (Z) [near end]\tikztonodes
				-| ([xshift=-2ex]\tikztotarget.west)
				-- (\tikztotarget)
			}
			]{dl} \\
			H_{n-1}(\T^{\cL' \Cup U}(E^n),\T^{\cL' \uplus U}(E^n))  & H_{n-1}(\CT^{\cL' \cap U}(U),\T^{\cL' \cap U}(U)) \lar[swap]{\cong},
		\end{tikzcd}
	\end{equation}
	where the first isomorphism takes a relative chain to its boundary, the second map is induced by a map of pairs, and the third isomorphism is by excision using that the square in \eqref{eqn:fact-pushout} is a pushout. If we apply this to the relative chain represented by an apartment-like map $(Q,\partial Q) \to \smash{(\CT^{\cL' \Cup U}(E^n),\T^{\cL' \Cup U}(E^n))}$, it restricts to $\partial Q$, then rewrites the resulting cycle as a relative cycle on the smaller pair $\smash{(\CT^{\cL' \cap U}(U),\T^{\cL' \cap U}(U))}$.

	Using this, we compute the dashed map of \eqref{st_en_map_of_les} on an element $[P] \in \smash{\Pt^{\cL' \cap U}(U)}$. Without loss of generality $P$ is a convex $(\cL' \cap U)$-polytope in $U$ that is a facet of a convex $(\cL' \Cup U)$-polytope $Q \subseteq U^+$ so that by \cref{pt_quotient}, $[Q] \in \smash{\Pt^{\cL' \Cup U}(E^n)}$ is a lift of $[P]$. We then map $[Q]$ down to the chain represented by the map 
	\[(Q,\partial Q) \longrightarrow (\CT^{\cL' \Cup U}(E^n),\T^{\cL' \Cup U}(E^n))\] 
	given by any apartment-like map on $\partial Q$. We now observe that apartment-like maps give a commuting diagram of pairs as follows:
	\[
	\begin{tikzcd}
		(\partial Q, \varnothing) \dar \rar{\apt} & (\T^{\cL' \Cup U}(E^n),\varnothing) \dar \\[-5pt]
		(\partial Q, \partial Q \setminus \textup{Int }P) \rar{\apt} & (\T^{\cL' \Cup U}(E^n),\T^{\cL' \uplus U}(E^n)) \\[-5pt]
		(P,\partial P) \uar{\cong_{H_*}} \rar{\apt} & (\CT^{\cL' \cap U}(U),\T^{\cL' \cap U}(U)), \uar[swap]{\cong_{H_*}}
	\end{tikzcd}
	\]
	where the maps marked $\cong_{H_*}$ are isomorphisms on homology. Furthermore, our convention for the orientation of $U$ implies that the distinguished generator of $H_n(Q,\partial Q)$ is sent to the distinguished generator of $H_{n-1}(P,\partial P)$. Therefore under \eqref{snake_map_expanded} the class $[Q]$ goes to the apartment class as desired
	\[ \smash{\apt[P] \in \widetilde{H}_{n-1}(\ST^{\cL' \cap U}(U))}.\]

	Now that we know the dashed map of the \eqref{st_en_map_of_les} agrees with the apartment map for the geometry $U$, it is an isomorphism by induction, so $\smash{\apt \colon \Pt^{\cL \Cup U}(E^n) \to \widetilde H_n(\ST^{\cL \Cup U}(E^n))}$ is an isomorphism using the 5-lemma.

	\smallskip
	
	For the third and final step, we prove that $\mathcal{P}_\cL$ has upper bounds for chains. We claim that given a totally ordered non-empty subset $\cdots \subset \cL_\lambda \subset \cL_{\lambda'} \subset \cdots$ of collections in $\mathcal{P}_\cL$, we have $\cup_{\lambda} \cL_\lambda \in \mathcal{P}_\cL$ as well. This union is evidently generated by hyperplanes and admissible, so it remains to check the conclusion of the theorem holds for it. We first observe that we have a homeomorphism
	\[\underset{\lambda \in \Lambda}\colim\,\T^{\cL_\lambda}(E^n) \overset{\cong}\longrightarrow \T^{\cup_{\lambda \in \Lambda} \cL_\lambda}(E^n) \]
	with domain a directed colimit of inclusions of cell complexes. Since any map from a compact space into $\T^{\cup_{\lambda \in \Lambda} \cL_\lambda}(E^n)$ factors over a $\T^{\cL_\lambda}(E^n)$, we deduce that the former is $(n-2)$-connected because each of the latter is. Since $\T^{\cup_{\lambda \in \Lambda} \cL_\lambda}(E^n)$ is also an $(n-1)$-dimensional complex, we conclude that it is equivalent to a wedge of $(n-1)$-spheres, proving the first half of the theorem. We next form the commutative diagram
	\[\begin{tikzcd} \underset{\lambda \in \Lambda}\colim\, \Pt^{\cL_\lambda}(E^n) \rar{\cong}[swap]{\apt} \dar[swap]{\cong} & \underset{\lambda \in \Lambda}\colim\, \widetilde{H}_n(\ST^{\cL_\lambda}(E^n)) \dar{\cong} \\[-5pt]
	\Pt^{\cup_{\lambda \in \Lambda} \cL_\lambda}(E^n) \rar[swap]{\apt} & \widetilde{H}_n(\ST^{\cup_{\lambda \in \Lambda} \cL_\lambda}(E^n)).
	\end{tikzcd}\]
	The right-hand vertical map is an isomorphism since homology commutes with directed colimits, and  the top horizontal map is a colimit of isomorphisms so an isomorphism. The left-hand vertical map is an isomorphism using the presentation of the polytope groups in \cref{def:polytope-group}, or alternatively using \cref{pt_subgroup}. Hence the bottom horizontal map is an isomorphism as well.
	
	\smallskip

	Having completed the three steps we can apply Zorn's lemma and conclude that $\mathcal{P}_\cL$ has a maximal element. This must be $\cL$, because if it were not then by the inductive step we could add another hyperplane. Therefore the conclusion of the theorem holds for $\cL$, as desired.
	\end{proof}
	
	\begin{remark}\label{n0-once-suspended}
		When $n = 0$, a ``once-suspended'' form of \cref{main-en} is true. We declare that $\cL = \{E^0\}$ is admissible, we see that $\T^\cL(E^0) = \varnothing$, and the apartment map for the unique polytope $* \subset E^0$ is the map $\varnothing \to \varnothing$. Therefore the suspension $\ST^\cL(E^0)$ is a single 0-sphere and the suspended apartment map induces an isomorphism
		\[ \begin{tikzcd} \Z \cong \Pt^\cL(E^0) \rar{\apt}[swap]{\cong} & \widetilde H_{0}(\ST^\cL(E^0)) \cong \Z. \end{tikzcd} \]
	\end{remark}
	
	\section{A local Solomon-Tits theorem and the hyperbolic case}
	
	In this section we prove a local version of the Solomon-Tits theorem that applies only to those hyperplanes that touch a fixed polytope $A$. We then use this to deduce a Solomon-Tits theorem when $\cL$ is generated by hyperplanes in the case of hyperbolic geometry.
	
	\subsection{The local Tits complex and local apartments}
	Let $X^n$ be either Euclidean or hyperbolic geometry, or an open hemisphere in spherical geometry, and let $A$ be a convex polytope in $X^n$. We will establish the Solomon-Tits theorem for $\cL$-polytopes contained inside $A$, by a similar induction on the number of hyperplanes in $\cL$ as \cref{main-en}. When intersecting with a hyperplane $U$ and reducing to the previously-proven version of the theorem inside $U$, we sometimes find that $A \cap U$ has dimension lower than $U$. To handle this, we also prove a version of the theorem where $A$ is allowed to have dimension smaller than $n$.
	
	\begin{definition}\label{def:polytope-group-q} 
		Let $\cL$ be a collection of subspaces of $X^n$ that is generated by hyperplanes, and let $V \in \cL$ be a subspace of any dimension. As in \cref{cup_and_cap_def}, let $(\cL \cap V) \subseteq \cL$ be the set of subspaces in $\cL$ that are contained in $V$, including $V$ itself. Finally, let $A$ be a non-empty convex $(\cL \cap V)$-polytope in $V$.
		\begin{enumerate}
			\item $\cL|A$ is the subposet of $\cL$ consisting of all subspaces that intersect $A$ in at least one point. The corresponding subcomplex of $T^\cL(X^n)$ is denoted $\T^{\cL|A}(X^n)$.
			\item $\Pt^{\cL|A}(X^n)$ is the subgroup of $\Pt^\cL(X^n)$ generated by those $\cL$-polytopes that are contained in $A$. (If the dimension of $A$ is less than $n$, this group is zero.)
		\end{enumerate}
	\end{definition} 
	
	\vspace{1em}
	\centerline{
	\def\svgwidth{1.9in}
\begingroup%
  \makeatletter%
  \providecommand\color[2][]{%
    \errmessage{(Inkscape) Color is used for the text in Inkscape, but the package 'color.sty' is not loaded}%
    \renewcommand\color[2][]{}%
  }%
  \providecommand\transparent[1]{%
    \errmessage{(Inkscape) Transparency is used (non-zero) for the text in Inkscape, but the package 'transparent.sty' is not loaded}%
    \renewcommand\transparent[1]{}%
  }%
  \providecommand\rotatebox[2]{#2}%
  \newcommand*\fsize{\dimexpr\f@size pt\relax}%
  \newcommand*\lineheight[1]{\fontsize{\fsize}{#1\fsize}\selectfont}%
  \ifx\svgwidth\undefined%
    \setlength{\unitlength}{97.33736693bp}%
    \ifx\svgscale\undefined%
      \relax%
    \else%
      \setlength{\unitlength}{\unitlength * \real{\svgscale}}%
    \fi%
  \else%
    \setlength{\unitlength}{\svgwidth}%
  \fi%
  \global\let\svgwidth\undefined%
  \global\let\svgscale\undefined%
  \makeatother%
  \begin{picture}(1,0.71508434)%
    \lineheight{1}%
    \setlength\tabcolsep{0pt}%
    \put(0,0){\includegraphics[width=\unitlength,page=1]{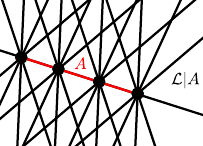}}%
  \end{picture}%
\endgroup%

	}
	\vspace{1em}
	
	Note that $\cL|A$ might no longer be closed under intersection: it is possible for $L_1 \cap A$ and $L_2 \cap A$ to be non-empty while $(L_1 \cap L_2) \cap A$ is empty. For example, in the figure above, any two black lines in $\cL|A$ that do not intersect in $A$, do not have an intersection that is in $\cL|A$.
	
	\begin{definition}\label{apt-restricted}
	If $P$ is a polytope contained in $A$, we say that a map $\partial P \to \T^{\cL|A}(X^n)$ is a \emph{restricted apartment-like map} if it sends each face $F$ into the contractible subcomplex $\CT^{\cL|A}(\spa F)$. These form a weakly contractible space as before, and coning off the restricted apartment-like maps gives maps of pairs $(P,\partial P) \to (\CT^{\cL|A}(X^n),\T^{\cL|A}(X^n))$ which can be used to define a homomorphism
	\begin{equation}\label{eq:apt_a}
		\apt \colon \Pt^{\cL|A}(X^n) \longrightarrow \widetilde{H}_n(\ST^{\cL|A}(X^n)).
	\end{equation}
	\end{definition}
	
	We first prove that when $\dim A < n$, this is an isomorphism between trivial groups.
	
	\begin{lemma}\label{q_complex_contractible}
		If $A$ is non-empty and $\dim A < n$, then $\T^{\cL|A}(X^n)$ is contractible and therefore \eqref{eq:apt_a} is an isomorphism.
	\end{lemma}

	\begin{proof}
		Let $V \in \cL$ be the span of $A$, and let $\CT^{(\cL \cap V)|A}(V)$ be the subcomplex of  $\T^{\cL|A}(X^n)$ given by those subspaces in $\cL \cap V$ that intersect $A$, including $V$ itself. This is contractible since $V$ is terminal. We define a deformation retraction onto this subcomplex
		\begin{align*}
			\T^{\cL|A}(X^n) &\longrightarrow \CT^{(\cL \cap V)|A}(V)
		\end{align*}
		by realising the map of posets $U \mapsto U \cap V$: this is an order-decreasing map, $U \supseteq f(U)$, that is the identity on $\CT^{(\cL \cap V)|A}(V)$ and preserves the property of having non-trivial intersection with $A$, as $(U \cap V) \cap A = U \cap A$. We conclude that $\smash{\T^{\cL|A}(X^n)}$ is contractible.
	\end{proof}
	
	We next move to the case where $\dim A = n$ and give the analogue of \cref{new_generators}. We note that in the subsequent, local version of this lemma $U$ can be any hyperplane in $X^n$---that is, the condition $U \notin \cL$ in \cref{new_generators} can be dropped.
	
	\begin{lemma}\label{new_generators_hn}
		If $\cL$ is generated by hyperplanes, $A$ is an $n$-dimensional convex $\cL$-polytope, and $U \subseteq X^n$ is any hyperplane, then each $(n-1)$-dimensional convex $(\cL \cap U)$-polytope $P \subseteq A \cap U$ is a facet of an $n$-dimensional convex $(\cL \Cup U)$-polytope $Q \subseteq A$.
	\end{lemma}

	\begin{proof}
		This is easier than \cref{new_generators} because we have the polytope $A$ to work with. Take the hyperplanes defining $A$ and $P$, including the hyperplane $U$. These subdivide $A$ into convex polytopes $Q_i$, and $P$ is either the facet of a unique $Q_i$ (when $U$ defines a facet of $A$) or of two polytopes $Q_i$ and $Q_j$ (when $U$ does not define a facet of $A$). Since each $Q_i$ is a convex $n$-dimensional $(\cL \Cup U)$-polytope, we are done.
	\end{proof}

	\vspace{1em}
	\centerline{
	\def\svgwidth{3.8in}
\begingroup%
  \makeatletter%
  \providecommand\color[2][]{%
    \errmessage{(Inkscape) Color is used for the text in Inkscape, but the package 'color.sty' is not loaded}%
    \renewcommand\color[2][]{}%
  }%
  \providecommand\transparent[1]{%
    \errmessage{(Inkscape) Transparency is used (non-zero) for the text in Inkscape, but the package 'transparent.sty' is not loaded}%
    \renewcommand\transparent[1]{}%
  }%
  \providecommand\rotatebox[2]{#2}%
  \newcommand*\fsize{\dimexpr\f@size pt\relax}%
  \newcommand*\lineheight[1]{\fontsize{\fsize}{#1\fsize}\selectfont}%
  \ifx\svgwidth\undefined%
    \setlength{\unitlength}{250.3603052bp}%
    \ifx\svgscale\undefined%
      \relax%
    \else%
      \setlength{\unitlength}{\unitlength * \real{\svgscale}}%
    \fi%
  \else%
    \setlength{\unitlength}{\svgwidth}%
  \fi%
  \global\let\svgwidth\undefined%
  \global\let\svgscale\undefined%
  \makeatother%
  \begin{picture}(1,0.41150674)%
    \lineheight{1}%
    \setlength\tabcolsep{0pt}%
    \put(0,0){\includegraphics[width=\unitlength,page=1]{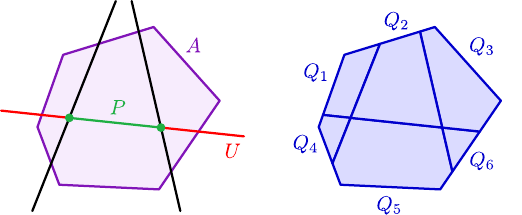}}%
  \end{picture}%
\endgroup%

	}
	\vspace{1em}
	
	Next we consider how $\Pt^{\cL|A}(X^n)$ changes when we add a new hyperplane $U$, analogously to \cref{pt_quotient}.
	
	\begin{lemma}\label{pt_quotient_hyp}
		If $\cL$ is generated by hyperplanes, $A$ is an $n$-dimensional convex $\cL$-polytope, and $U \notin \cL$ is a hyperplane, then there is an exact sequence of abelian groups
		\[ 0 \longrightarrow \Pt^{\cL|A}(X^n) \longrightarrow \Pt^{(\cL \Cup U)|A}(X^n) \longrightarrow \Pt^{(\cL \cap U)|(A \cap U)}(U) \longrightarrow 0. \]
	\end{lemma}
	
	It is possible for $A \cap U$ to have dimension less than $(n-1)$, or to be empty. In either of these cases $\Pt^{(\cL \cap U)|(A \cap U)}(U) = 0$ and we shall show the left-hand map is an isomorphism.
	
	\begin{proof}
		The proof is analogous to that of \cref{pt_quotient}. In particular, we label the half-spaces cut out by $U$ as $U^+,U^- \subseteq X^n$, and define the map
		\begin{equation}\label{pt_quotient_identification-hn}
			\Pt^{(\cL \Cup U)|A}(X^n) \big / \Pt^{\cL|A}(X^n) \longrightarrow \Pt^{(\cL \cap U)|(A \cap U)}(U)
		\end{equation}
		by taking each convex $(\cL \Cup U)$-polytope $Q \subseteq A$ to its facet (if any) that lies along $U$, with positive sign if $Q \subseteq U^+$ and negative sign if $Q \subseteq U^-$. The proof that this is an isomorphism is the same as in \cref{pt_quotient}, using only polytopes that are contained in $A$ and using \cref{new_generators_hn} when defining the inverse. Note that when $A \cap U$ is not $n$-dimensional, $U$ does not intersect the interior of $A$, and so there is no difference between convex $\cL$-polytopes contained in $A$ and convex $(\cL \Cup U)$-polytopes contained in $A$.
	\end{proof}
	
	\begin{theorem}[$A$-local Solomon-Tits]\label{main-step-hn}
		For all $n \geq 1$, if $\cL$ is generated by hyperplanes and $A$ is an $n$-dimensional convex $\cL$-polytope, then $\T^{\cL|A}(X^n)$ is equivalent to a wedge of $(n-1)$-spheres, and the restricted apartment map from \cref{apt-restricted} induces an isomorphism
		\[\apt \colon \Pt^{\cL|A}(X^n) \overset{\cong}\longrightarrow \widetilde{H}_n(\ST^{\cL|A}(X^n)). \]
	\end{theorem}
	
	\begin{proof}
	As for \cref{main-en}, the proof is by induction on $n$. In the initial case $n = 1$, $A$ is a line segment with endpoints in $\cL^0$, and $\T^{\cL|A}(X^1)$ is set of points of $\cL^0$ contained between these endpoints. Again its suspension is a wedge of circles, and the map
		\[\apt \colon \Pt^{\cL|A}(X^1) \longrightarrow \widetilde{H}_1(\ST^{\cL|A}(X^1)) \cong \ker(\Z\langle \cL^0 \cap A \rangle \overset{\epsilon}\longrightarrow \Z)\]
	sends each interval to the difference between its endpoints $[p'] - [p]$, which is easily checked to be an isomorphism.
	
	We henceforth assume $n \geq 2$ and that the theorem has been proven for the previous value of $n$. As before, we prove the result by applying Zorn's Lemma to the set $\mathcal{P}_\cL$ of all collections $\cL' \subseteq \cL$ that are generated by hyperplanes and for which the theorem is true, now for a fixed polytope $A$. Note that such a collection $\cL'$ must at the very least contain the hyperplanes that define $A$ itself. There are three main steps:
	\begin{itemize}
		\item proving that $\mathcal{P}_\cL$ is non-empty,
		\item adding one hyperplane $U$ to $\cL' \in \mathcal{P}_\cL$, and proving that $\cL' \Cup U \in \mathcal{P}_\cL$, and
		\item taking a colimit of collections $\cup_{\lambda} \cL_\lambda$ with $\cL_\lambda \in \mathcal{P}_\cL$ and proving that $\cup_{\lambda} \cL_\lambda \in \mathcal{P}_\cL$.
	\end{itemize}
	
	For the first step, we prove that $\mathcal{P}_\cL$ is non-empty. Let $\cL_\circ \subseteq \cL$ be the minimal collection of hyperplanes that define $A$, and their intersections. By \cref{faces_to_L}, this has one hyperplane for each facet of $A$. Each of the remaining subspaces $U \in \cL_\circ|A$ intersects $A$ in a face, because \cref{faces_to_L} implies the intersection property
	\[ A \cap (\spa F_1 \cap \spa F_2) = (A \cap \spa F_1) \cap (A \cap \spa F_2) = F_1 \cap F_2. \]
	This gives a surjection from $\cL_\circ|A$ to the nonempty faces of $A$. It is not in general a bijection because the same face could be obtained as the intersection of several subspaces in $\cL_\circ$ with $A$. (For instance, if $A$ is an octahedron and $F$ is a vertex, then there are three subspaces in $\cL_\circ$ whose intersection with $A$ is the vertex.)
	
	Still, the rule $U \mapsto U \cap \spa(U \cap A)$ gives a deformation retraction of the complex $\T^{\cL_\circ|A}(X^n)$ onto the subcomplex of those spaces that are spans of the faces of $A$. This is identified with the derived subdivision of $\partial A$, in other words the simplicial complex with a simplex for each flag of proper nonempty faces $F_0 \subseteq \cdots \subseteq F_p$. Furthermore the inclusion $\partial A \to \T^{\cL_\circ|A}(X^n)$ is a restricted apartment-like map. So the entire complex is homotopy equivalent to a single $(n-1)$-sphere, $\Pt^{\cL_\circ|A}(X^n) = \Z$, and the restricted apartment classes induce an isomorphism on homology.
	
	\smallskip
	
	For the second key inductive step, we form the diagram
	\begin{equation}\label{eqn:fact-pushout-hn}
	\begin{tikzcd}
		\T^{\cL'|A}(X^n) \rar{\simeq} &  \T^{(\cL' \uplus U)|A}(X^n) \rar & \T^{(\cL' \Cup U)|A}(X^n) \\[-5 pt]
		& \T^{(\cL' \cap U)|(A \cap U)}(U) \rar \uar & \CT^{(\cL' \cap U)|(A \cap U)}(U). \uar
	\end{tikzcd}
	\end{equation}
	The map marked $\simeq$ is an equivalence using the same deformation retract given below the diagram \eqref{eqn:fact-pushout}. This deformation retract preserves the property of having non-trivial intersection with $A$, because it sends each space $V \in \cL' \uplus U$ to a larger space in $\cL'$. The remaining square of \eqref{eqn:fact-pushout-hn} is a pushout because, as before, a space in $\cL' \Cup U \setminus \{X^n\}$ that intersects $A$ either has to be a subspace of $U$, or an element of $\cL' \uplus U \setminus \{X^n\}$, or both, which happens precisely when it is an element of $(\cL' \cap U) \setminus \{U\}$. In all of these cases, the space must continue to intersect $A$.
	
	If $U$ does not intersect $A$, then the complexes on the bottom row of \eqref{eqn:fact-pushout-hn} are empty, the ones along the top row are identical, and the polytope group does not change by \cref{pt_quotient_hyp}, so $\cL' \Cup U \in \mathcal{P}_\cL$ as desired. 
	
	If $U$ intersects $A$ but $\dim(A \cap U) < \dim U$, then the complexes on the bottom row of \eqref{eqn:fact-pushout-hn} are contractible by \cref{q_complex_contractible}, so the inclusion $\T^{\cL'|A}(X^n) \to \T^{(\cL' \Cup U)|A}(X^n)$ is an equivalence, and so $\T^{(\cL' \Cup U)|A}(X^n)$ is a wedge of $(n-1)$-spheres. The map of polytope groups $\Pt^{\cL'|A}(X^n) \to \Pt^{(\cL' \Cup U)|A}(X^n)$ is also an isomorphism by \cref{pt_quotient_hyp}, so the apartment classes still give an isomorphism to $\widetilde{H}_n(\ST^{(\cL' \Cup U)|A}(X^n))$, and so again $\cL' \Cup U \in \mathcal{P}_\cL$.
	
	If $U$ intersects $A$ and $\dim(A \cap U) = \dim U = (n-1)$, then $U$ is not one of the hyperplanes defining $A$ (those are already contained in $\cL'$), so $U$ intersects the interior of $A$. This is the interesting case, because new homology classes are created. In this case the bottom-right term $\smash{\CT^{(\cL' \cap U)|(A \cap U)}(U)}$ is contractible, and the bottom-left term $\smash{\T^{(\cL' \cap U)|(A \cap U)}(U)}$ is a wedge of $(n-2)$-spheres by inductive hypothesis. Therefore $\smash{\T^{(\cL' \Cup U)|A}(X^n)}$ is a wedge of $(n-1)$-spheres. We form the map of exact sequences
	\begin{equation}\label{st_hn_map_of_les}
	\begin{tikzcd}[column sep = .9em]
		0 \rar & \Pt^{\cL'|A}(X^n) \rar \dar{\apt}[swap]{\cong} & \Pt^{(\cL' \Cup U)|A}(X^n) \dar{\apt} \rar & \Pt^{(\cL' \cap U)|(A \cap U)}(U) \dar[dashed] \rar & 0 \\[-5pt]
		0 \rar & \widetilde{H}_n(\ST^{\cL'|A}(X^n)) \rar &  \widetilde{H}_n(\ST^{(\cL' \Cup U)|A}(X^n)) \rar & \widetilde{H}_{n-1}(\ST^{(\cL' \cap U)|(A \cap U)}(U)) \rar & 0,\end{tikzcd}
	\end{equation}
	 and the same argument as before shows that the dashed map is given by the restricted apartment classes in $U$. (The lift $[Q]$ of $[P]$ has to be chosen so that $Q \subseteq A$, which is possible by \cref{new_generators_hn}.) Therefore, the dashed map is an isomorphism by inductive hypothesis. Hence, the middle vertical map is an isomorphism, and we conclude that $\cL' \Cup U \in \mathcal{P}_\cL$ in this case as well. This finishes all the cases and therefore finishes the second step of the proof.
	
	\smallskip
	
	The proof of the third step (the one that takes the colimit $\cup_\lambda \cL_\lambda$), as well as the end of the proof, is identical to the proof in \cref{main-en}.
	\end{proof}
	
	\subsection{Proof in the hyperbolic case}
	We will now prove the Solomon-Tits theorem in the hyperbolic case when $\cL$ is generated by hyperplanes, using the above local form of the Solomon-Tits theorem. In order for this approach to work, we need a stronger notion of admissibility:

	\begin{definition}\label{admissible-hn}
	Let $n \geq 1$. We say that a collection $\cL$ of subspaces of $H^n$ is \emph{admissible} if
		\begin{enumerate}
			\item $\cL$ is generated by hyperplanes, and 
			\item there is an increasing sequence of convex $\cL$-polytopes $A_1 \subseteq A_2 \subseteq A_3 \subseteq \cdots$ such that the union of the interiors is all of $H^n$.
		\end{enumerate}
	\end{definition}
	
	\begin{remark}
		Condition (ii) is stronger than the one we imposed in the Euclidean case: that $\cL^0$ is non-empty (\cref{admissible-en}). The Solomon-Tits theorem for $H^n$ can be false under that weaker assumption: take seven hyperplanes in $E^3$, whose intersections form two cubes sharing a common face, and arrange these hyperplanes so the two cubes lie mostly in the unit ball $B^3 \subseteq E^3$, except that exactly two vertices of one of the cubes stick outside the ball, as shown below. Then in the Klein model for $H^3$, this gives seven hyperplanes in $H^3$ with non-trivial 0-dimensional intersections. However, $\T^\cL(H^3)$ has the homotopy type of $S^1 \vee S^2$.
	\end{remark}

	\vspace{1em}
	\centerline{
	\def\svgwidth{1.8in}
\begingroup%
  \makeatletter%
  \providecommand\color[2][]{%
    \errmessage{(Inkscape) Color is used for the text in Inkscape, but the package 'color.sty' is not loaded}%
    \renewcommand\color[2][]{}%
  }%
  \providecommand\transparent[1]{%
    \errmessage{(Inkscape) Transparency is used (non-zero) for the text in Inkscape, but the package 'transparent.sty' is not loaded}%
    \renewcommand\transparent[1]{}%
  }%
  \providecommand\rotatebox[2]{#2}%
  \newcommand*\fsize{\dimexpr\f@size pt\relax}%
  \newcommand*\lineheight[1]{\fontsize{\fsize}{#1\fsize}\selectfont}%
  \ifx\svgwidth\undefined%
    \setlength{\unitlength}{124.29032789bp}%
    \ifx\svgscale\undefined%
      \relax%
    \else%
      \setlength{\unitlength}{\unitlength * \real{\svgscale}}%
    \fi%
  \else%
    \setlength{\unitlength}{\svgwidth}%
  \fi%
  \global\let\svgwidth\undefined%
  \global\let\svgscale\undefined%
  \makeatother%
  \begin{picture}(1,0.94282322)%
    \lineheight{1}%
    \setlength\tabcolsep{0pt}%
    \put(0,0){\includegraphics[width=\unitlength,page=1]{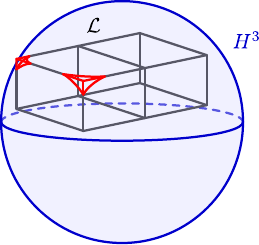}}%
  \end{picture}%
\endgroup%

	}
	\vspace{1em}

	When $n \geq 2$, admissibility in the sense of \cref{admissible-hn} inherited by $\cL \cap U$ for any hyperplane $U \in \cL^{n-1}$. The following lemma makes the condition easier to verify in examples.
	
	\begin{lemma}
		If $\cL$ is generated by hyperplanes, then it is admissible if and only if for every $x \in H^n$, every ray originating at $x$, and every $r > 0$, there exists a hyperplane in $\cL$ that intersects the ray and whose closest distance to $x$ in the hyperbolic metric is larger than $r$.
	\end{lemma}
	
	\begin{proof}
		Suppose we have the increasing sequence of polytopes $A_i$. Given a ray originating at $x$ and a ball $B_r$ of radius $r$ about $x$, some $A_i$ contains $B_r$ in its interior since $\overline{B_r}$ is compact. Then the ray intersects $\partial A_i$ at some point; taking one of the hyperplanes defining $A_i$ that passes through that point, since the hyperplane does not intersect the interior of $A_i$, it does not intersect $B_r$, as desired.
		
		Conversely, suppose we have the condition in the statement of the lemma and we fix $x \in H^n$. Given $A_{i-1}$, since it is bounded we can choose a large ball $B_i$ centred at $x$ and containing $A_{i-1}$, of radius at least $i$. For every ray originating from $x$, we may choose a hyperplane in $\cL$ intersecting the ray and not intersecting $B_i$. Necessarily the hyperplane must intersect the ray transversally, otherwise they would be tangent and that would force the hyperplane to pass through $x$.
		
		Since the intersection is transverse, the same hyperplane will also intersect all rays through $x$ sufficiently close to the given ray. Therefore by compactness of $S^{n-1}$, it is possible to select finitely many hyperplanes that collectively intersect all possible rays emanating from $x$, and that individually do not intersect $B_i$. It follows that these hyperplanes define a convex $\cL$-polytope $A_i$ containing $B_i$ and therefore also containing $A_{i-1}$. Since the radii of the balls $B_i$ tend to infinity, the union of the polytopes $A_i$ is all of $H^n$.
	\end{proof}
	
	Now we will prove the Solomon-Tits theorem for $\T^\cL(H^n)$ assuming $\cL$ is admissible.

	\begin{theorem}[Solomon-Tits for $\cL$-polytopes in $H^n$]\label{main-hn}
		For all $n \geq 1$, if $\cL$ is admissible then $\T^\cL(H^n)$ is equivalent to a wedge of $(n-1)$-spheres and the apartment classes of \cref{apartment_pt} induce an isomorphism
		\[ \Pt^\cL(H^n) \overset{\cong}\longrightarrow \widetilde H_{n}(\ST^\cL(H^n)). \]
	\end{theorem}

	\begin{proof}
		Take a sequence of convex $\cL$-polytopes $A_1 \subseteq A_2 \subseteq A_3 \subseteq \cdots$ whose union is all of $H^n$. As in the proof of \cref{main-en}, we have a homeomorphism
		\[\underset{i \to \infty}\colim\,\T^{\cL|A_i}(H^n) \overset{\cong}\longrightarrow \T^{\cL}(H^n). \]
		The complexes $\T^{\cL|A_i}(H^n)$ are all wedges of $(n-1)$-spheres by \cref{main-step-hn}, hence are $(n-2)$-connected and so is their colimit $\smash{\T^{\cL}(H^n)}$. Since the colimit is also $(n-1)$-dimensional, it is therefore equivalent to a wedge of $(n-1)$-spheres, proving the first half of the theorem.
		
		For the second half, we form the commutative diagram
		\[\begin{tikzcd} \underset{i \to \infty}\colim\, \Pt^{\cL|A_i}(H^n) \rar{\cong}[swap]{\apt} \dar[swap]{\cong} &[5pt] \underset{i \to \infty}\colim\, \widetilde{H}_n(\ST^{\cL|A_i}(H^n)) \dar{\cong} \\[-5pt]
		\Pt^{\cL}(H^n) \rar[swap]{\apt} & \widetilde{H}_n(\ST^{\cL}(H^n)).
		\end{tikzcd}\]
		in which the horizontal maps are given by the apartment class map and the restricted apartment class map. As in the proof of \cref{main-en}, the right-hand vertical map is an isomorphism since homology commutes with directed colimits, the top horizontal map is a colimit of isomorphisms by \cref{main-step-hn} so an isomorphism, and the left-hand vertical map is an isomorphism by the presentation of the polytope groups. Hence the bottom horizontal map is an isomorphism.
	\end{proof}
	
	\begin{remark}
		As in \cref{n0-once-suspended}, a ``once-suspended'' form of this theorem holds when $n = 0$. In fact, it is the same theorem as the one stated in \cref{n0-once-suspended}, because $H^0$ is isomorphic to $E^0$. In particular, we have an isomorphism $\apt \colon \Z \cong \Pt^\cL(H^0) \smash{\overset{\cong}\longrightarrow} \widetilde H_{0}(\ST^\cL(H^0)) \cong \Z$.
	\end{remark}

	\section{Apartments in \texorpdfstring{$\PT^\cL(S^n)$}{PT\unichar{"005E}\unichar{"1D4DB}(S\unichar{"005E}n)} and the spherical case}
	
	In this section we prove \cref{thm:generated-by-hyperplanes} in the case of spherical geometry. Though the idea of the argument is the same, the details require substantial modifications because the complexes $\PT^\cL(S^n)$ and $\ST^\cL(S^n)$ are not equivalent, see the comments following their introduction in \cref{df:st-and-pt}. 
	
	In particular, the definition of an apartment-like map $\partial P \to \T^\cL(S^n)$ has to be modified so as to get a map $P/\partial P \to \PT^\cL(S^n)$ instead of a map to $\ST^\cL(S^n)$. It is also more productive to think of $\cL$ as a collection of linear subspaces of $\R^{n+1}$, rather than as a collection of geodesic subspaces of $S^n$, although of course these are equivalent.
	
	\subsection{Admissibility and strongly convex polytopes}
	
	When we are working with $\PT^\cL$, it is convenient to think of $\cL$ as a collection of linear subspaces of $\R^{n+1}$. Doing so we can allow the zero subspace, even though we did not allow empty subsets of our geometry before, and this does not affect the definition of $\PT^\cL(S^n)$ since empty subspaces do not contribute to the homotopy colimit. This is only really important in the 0-dimensional case $S^0$, since it affects whether the individual points in $S^0$ are $\cL$-polytopes.
	
	\begin{definition}\label{admissible-sn}
	Let $n \geq 0$. We say that a collection $\cL$ of subspaces of $\R^{n+1}$ is \emph{admissible} if
		\begin{enumerate}
			\item $\cL$ is generated by hyperplanes, and 
			\item the intersection of all the hyperplanes in $\cL$ is the zero vector $\{0\}$.
		\end{enumerate}
	\end{definition}
	
	When $n = 0$, the only admissible collection of hyperplanes is the collection whose only element is the zero subspace $\{0\} \subseteq \R^1$. In particular, we have $\Pt^{\cL}(S^0) \cong \Z^2$ for this collection in contrast to the result for the collection with no hyperplanes in \cref{rem:spherical-case-polytope-group}.

	When $\cL$ is admissible, the associated geodesic subspaces of $S^n$ are also generated by hyperplanes, and the intersection of all of them is empty. For the purposes of the proofs in this section, we will think of the unit sphere as a quotient of the complement of the origin rather than a subspace, as this will allow us to apply linear changes of coordinates in our proofs:
	\[ S^n = (\R^{n+1} \setminus \{0\})/\R_+. \]
	
	\begin{lemma}\label{admissible-contains-coords}
		If $\cL$ is an admissible collection of subspaces of $\R^{n+1}$, then there exists a linear coordinate change after which $\cL$ contains all of the standard coordinate planes $\{ x \in \R^{n+1} : x_i = 0\}$.
	\end{lemma}
	
	\begin{proof}
		Select any hyperplane $U_1 \in \cL$, and then inductively choose a hyperplane $U_i \in \cL$ such that the intersection $U_1 \cap \cdots \cap U_i$ has dimension $(n+1)-i$, giving $(n+1)$ hyperplanes $U_1$, $\ldots$, $U_{n+1}$ whose intersection is the zero vector. Any intersection of $n$ of the hyperplanes must be 1-dimensional, and we take $v_i$ to be a vector spanning the line $U_1 \cap \cdots \cap \widehat{U_i} \cap \cdots \cap U_{n+1}$. Then in the coordinate system given by $v_1$, $\ldots$, $v_{n+1}$, the hyperplanes $U_1$, $\ldots$, $U_{n+1}$ become the standard coordinate hyperplanes.
	\end{proof}
	
	We will only be able to define apartments for convex polytopes that are contained in an open hemisphere. This condition will arise frequently, so we give it a special name.
	\begin{definition}\label{strong-convex}
		A polytope $P \subseteq S^n$ is \emph{strongly convex} if it is convex (a finite intersection of half-spaces) and contained in an open hemisphere.
	\end{definition}
	
	\begin{lemma}\label{strongly-convex-intersection}
		A convex polytope $P$ is strongly convex if and only if the hyperplanes that define $P$ have as their intersection the zero subspace of $\R^{n+1}$.
	\end{lemma}
	
	\begin{proof}
		For the ``if'' direction, suppose that the intersection of the hyperplanes defining $P$ is the zero subspace. Then, by \cref{admissible-contains-coords} and up to linear change of coordinates, $P$ is contained in a region of the form $\{ x \in \R^{n+1} \mid x_i \geq 0\}$, which is contained in an open hemisphere. Since any linear coordinate change preserves open hemispheres, $P$ itself is also contained in an open hemisphere, and hence strongly convex. The ``only if'' direction follows by observing that if the intersection of the hyperplanes defining $P$ is non-zero, then $P$ contains a pair of antipodal points, and hence does not lie in a single open hemisphere.
	\end{proof}
	
	\begin{lemma} 
		A collection of hyperplanes $\cL$ is admissible if and only if there exists a strongly convex $\cL$-polytope. If so, every $\cL$-polytope decomposes into strongly convex $\cL$-polytopes.
	\end{lemma}

	\begin{proof}
		For the ``only if'' direction, we note that if $\cL$ is admissible then after a linear coordinate change it contains the $2^{n+1}$ spherical simplices cut out by the $(n+1)$ standard coordinate hyperplanes in $\R^{n+1}$. For the addendum observe that each of these simplices is contained in an open hemisphere, and that any $\cL$-polytope can be cut along these hyperplanes to make pieces that are contained in an open hemisphere. The ``if'' direction follows by noting that if there exists a strongly convex $\cL$-polytope $P$, then $\cL$ is admissible by \cref{strongly-convex-intersection}.
	\end{proof}
	
	\begin{lemma}\label{strongly-convex-contractible}
		If $P$ is strongly convex then every face of $P$ is contractible.
	\end{lemma}
	
	\begin{proof}
		Embed $P$ in an open hemisphere and note that any two points in this hemisphere are connected by a unique geodesic. This can be used to construct deformation retraction from a face onto any one of its points.
	\end{proof}
	
	\subsection{Apartments in \texorpdfstring{$\PT^\cL(S^n)$}{PT\unichar{"005E}\unichar{"1D4DB}(S\unichar{"005E}n)}}
	
	Next we recall the definition of $\PT^{\cL}(S^n)$ from \cref{df:st-and-pt}. It will be convenient to write it in the following way. Use $V$ to denote any linear subspace of $\R^{n+1}$ and $S(V)$ the corresponding sphere $(V \setminus \{0\})/\R_+$ with positive real numbers acting by scaling. Then for non-empty $\cL$ the definition of $\PT^{\cL}(S^n)$ from \cref{df:st-and-pt} becomes
		\[ \PT^{\cL}(S^n) = \frac{ \underset{V \in \cL}\hocolim\, S(V) }{ \underset{V \in \cL \setminus \{\R^{n+1}\}}\hocolim\, S(V) },\]
	where once more we fix the concrete model for the homotopy colimits provided by the Bousfield--Kan formula. This is given by a geometric realisation of a simplicial object and hence gives us point-set topological control over the terms appearing in this quotient:
	
	\begin{lemma}\label{pt-subcomplex}
		The space $\hocolim_{V \in \cL} S(V)$ is an $n$-dimensional cell complex, and for every sub-poset $\cL' \subseteq \cL$ the subspace $\hocolim_{V \in \cL'} S(V)$ is a subcomplex.
	\end{lemma}
	
	\begin{proof}
		We fix some cell complex structure on $S(V)$ for every subspace $V$ and we define maps into the homotopy colimit of the form
		\[ \Delta^k \times D^m \to \underset{V \in \cL}\hocolim\, S(V), \]
		one for each $k$-tuple of composable inclusions of linear subspaces $\varnothing \subsetneq U_0 \subsetneq U_1 \subsetneq \cdots \subsetneq U_k \subseteq \R^{n+1}$ and choice of $m$-cell $D^m \to S(U_0)$. Note that the maximum dimension of any one of these cells is $k+\dim\,S(U_0)$, which is no greater than
		\[ ((n+1) - \dim U_0) + (\dim U_0 - 1) = n. \]
		To see that these give a cell complex structure, we recall the Bousfield--Kan formula expresses the homotopy colimit as a colimit of skeleta, and each skeleton is constructed from the previous one as a pushout
		\[ \begin{tikzcd}
			\coprod_{U_0 \subsetneq U_1 \subsetneq \cdots \subsetneq U_k} \partial \Delta^k \times S(U_0) \dar \rar &
			\coprod_{U_0 \subsetneq U_1 \subsetneq \cdots \subsetneq U_k} \Delta^k \times S(U_0) \dar \\[-5pt]
			\left(\underset{V \in \cL'}\hocolim\, S(V)\right)^{(k-1)} \rar &
			\left(\underset{V \in \cL'}\hocolim\, S(V)\right)^{(k)}.
		\end{tikzcd} \]
		Each of the maps $\partial\Delta^k \times S(U_0) \to \Delta^k \times S(U_0)$ is formed as an iterated pushout along the cells we constructed for a fixed choice of $U_0,\ldots,U_k$. Putting these all together, we deduce that the horizontal maps of the above diagram are formed as pushouts along the cells we chose for a fixed value of $k$. Therefore the entire homotopy colimit is formed as an iterated pushout along all of the cells that we chose.
	\end{proof}
		
	\begin{definition}\label{pt-apt-like}
		If $P$ is a strongly convex $\cL$-polytope in $S^n$, a map
		\[ P \longrightarrow \underset{V \in \cL}\hocolim\, S(V) \]
		is \emph{apartment-like} if for every face $F \subseteq P$ of any dimension, including $P$ itself, the image of $F$ lands in the subspace
		\begin{equation}\label{face_complex_spherical_case}
			N(F) := \underset{V \in (\cL \cap \spa F)}\hocolim\, (S(V) \cap F).
		\end{equation}
	\end{definition}
	
	\begin{lemma}\label{apt-like-preserved-3}
		If $P$ is strongly convex, then the space of apartment-like maps in the sense of \cref{pt-apt-like} is weakly contractible.
	\end{lemma}
	
	\begin{proof}
		The proof is the same as in \cref{apt-like-preserved} and \cref{apt-like-preserved-2}, except that we are mapping out of $D^m \times P$, and as in \cref{apt-like-preserved-2} we think of $P$ as a cell complex with one cell for each face. We only need to check that for the complex $N(F)$ defined in \eqref{face_complex_spherical_case}, an inclusion $F \subseteq F'$ implies $N(F) \subseteq N(F')$, and that furthermore $N(F)$ is contractible for each face $F$ of $P$.
		
		The first of these conditions follows from inspecting the definition of the homotopy colimit. For the second condition, the homotopy colimit in \eqref{face_complex_spherical_case} has a terminal object, namely the span of $F$, so it is homotopy equivalent to $S(\spa F) \cap F \cong F$. But $F$ is contractible by \cref{strongly-convex-contractible}, and therefore the homotopy colimit is contractible, as required.
	\end{proof}
	
	\begin{definition}\label{apartment_pt_sn}
		For each strongly convex polytope $P$, we define the \emph{suspended apartment map}
		\[ S^n \cong \frac{P}{\partial P} \longrightarrow \frac{ \underset{V \in \cL}\hocolim\, S(V) }{ \underset{V \in \cL \setminus \{\R^{n+1}\}}\hocolim\, S(V) } = \PT^\cL(S^n)\]
		as the map induced by any apartment-like map $P \to \hocolim_{V \in \cL} S(V)$. The resulting \emph{apartment class} is the image of the fundamental class under the induced map on homology
		\[ \apt[P] \in \widetilde{H}_n(\PT^\cL(S^n)). \]
		As in \cref{apartment_pt}, we fix an orientation on $S^n$, which induces an orientation on each sphere $P/\partial P$.
	\end{definition}
	
	Note that along the forgetful map $\PT^\cL(S^n) \to \ST^\cL(S^n)$, these apartment classes go to the apartment classes defined previously in \cref{apartment_pt}, justifying the overlap in terminology. The proof of the following is the same as in \cref{pt-apt-well-defined}:

	\begin{lemma}\label{pt-apt-well-defined-sn}
		When $\cL$ is admissible, the apartment-like maps induce a well-defined homomorphism
		\[\apt \colon \Pt^\cL(S^n) \longrightarrow \widetilde{H}_{n}(\PT^\cL(S^n)).\]
	\end{lemma}
	
	\subsection{Proof in the spherical case for \texorpdfstring{$\cL$}{\unichar{"1D4DB}} admissible}
	
	Now that we have introduced the new notion of apartment for $S^n$, we finish the remaining preliminaries and prove \cref{thm:generated-by-hyperplanes} in the case that $\cL$ is admissible.
	
	Let $U$ be any hyperplane in $\R^{n+1}$, and recall from \cref{cup_and_cap_def} that $\cL \cap U$ is the set of subspaces of the form $L \cap U$ for $L \in \cL$, while $\cL \Cup U$ is the set of intersections of all hyperplanes in the set $\cL^{n-1} \cup \{U\}$.
	
	\begin{lemma}\label{new_generators_sn}
		If $\cL$ is generated by hyperplanes and $U \subseteq \R^{n+1}$ is a hyperplane, then any $(n-1)$-dimensional convex $(\cL \cap U)$-polytope $P \subseteq S(U)$ is a facet of two convex $n$-dimensional $(\cL \Cup U)$-polytopes $Q,Q' \subseteq S^n$. Furthermore if $P$ is strongly convex then so are $Q$ and $Q'$.
	\end{lemma}

	\begin{proof}
		The proof is as in \cref{new_generators_hn} with $A = S^n$, so we only need to subdivide using $U$ and the hyperplanes that define $P$.
		If in addition $P$ is strongly convex in $U$, then any convex polytope $Q$ that has $P$ as a facet must also be strongly convex, using \cref{strongly-convex-intersection}.
	\end{proof}
	
	\begin{lemma}\label{pt_quotient_sph}
		If $\cL$ is generated by hyperplanes and $U \notin \cL$ is a hyperplane, then there is an exact sequence of abelian groups
		\[ 0 \longrightarrow \Pt^{\cL}(S^n) \longrightarrow \Pt^{\cL \Cup U}(S^n) \longrightarrow \Pt^{\cL \cap U}(S(U)) \longrightarrow 0. \]
	\end{lemma}

	\begin{proof}
		The proof is the same as in \cref{pt_quotient}, using \cref{new_generators_sn}.
	\end{proof}
	
	\begin{theorem}[Solomon-Tits for $\cL$-polytopes in $S^n$]\label{main-sn}
		For all $n \geq 0$, if $\cL$ is admissible  then $\PT^\cL(S^n)$ is equivalent to a wedge of $n$-spheres and the apartment classes of \cref{apartment_pt_sn} induce an isomorphism
		\[ \Pt^\cL(S^n) \overset{\cong}\longrightarrow \widetilde H_{n}(\PT^\cL(S^n)). \]
	\end{theorem}
	
	Note that the statement is slightly different than the versions for Euclidean and hyperbolic geometry (\cref{main-en} and \cref{main-hn}) because there is no wedge of $(n-1)$-spheres that we suspend to form $\PT^\cL(S^n)$.
	
	\begin{proof}
	As in \cref{main-en} and \cref{main-step-hn}, the proof is by induction on $n$. In the initial case $n = 0$, the only possible admissible collection of subspaces is $\cL = \{0,\R^1\}$. The zero subspace in $\cL$ cuts the 0-sphere into its two constituent points, so that $\Pt^\cL(S^n) = \Z^2$. Similarly we have
	\[ \PT^\cL(S^0) = \frac{S(\R^1)}{\varnothing} = S(\R^1)_+ \cong S^0 \vee S^0, \]
	which also has $\Z^2$ as its reduced 0th homology. Each of the two points in the 0-sphere is a strongly convex $\cL$-polytope, and if we let $P$ be one of these points, then the apartment-like map for $P$ is just the inclusion of $P$ into $S^0$. On quotients this induces a map $S^0 \to S^0 \vee S^0$ which includes one of the summands. Therefore it induces an isomorphism $\Z^2 \cong \Z^2$ on reduced 0th homology.
	
	We now assume $n \geq 1$ and that the theorem has been proven for the previous value of $n$. As before, we let $\mathcal{P}_\cL$ denote all admissible collections $\cL' \subseteq \cL$ for which the theorem is true, and follow three steps:
	\begin{itemize}
		\item proving that $\mathcal{P}_\cL$ is non-empty,
		\item adding one hyperplane $U$ to $\cL' \in \mathcal{P}_\cL$, and proving that $\cL' \Cup U \in \mathcal{P}_\cL$, and
		\item taking a colimit of collections $\cup_{\lambda} \cL_\lambda$ with $\cL_\lambda \in \mathcal{P}_\cL$ and proving that $\cup_{\lambda} \cL_\lambda \in \mathcal{P}_\cL$.
	\end{itemize}
	
	For the first step, we let $\cL_\circ \subseteq \cL$ be the collection generated by the coordinate hyperplanes from \cref{admissible-contains-coords}. Without loss of generality they are actually the standard coordinate hyperplanes. They cut $S^n$ into $2^{n+1}$ simplices which we label $P_1$ through $P_{2^{n+1}}$, from which we get that $\Pt^{\cL_\circ}(S^n)$ is free on $2^{n+1}$ generators. Similarly we can form an equivalence
	\begin{equation}\label{eq:pt-collapse}
		\PT^{\cL}(S^n) = \frac{ \underset{V \in \cL}\hocolim\, S(V) }{ \underset{V \in \cL \setminus \{\R^{n+1}\}}\hocolim\, S(V) } \overset{\simeq}\longrightarrow \frac{ S(\R^{n+1}) }{ \bigcup_{V \in \cL_\circ \setminus \{\R^{n+1}\}} S(V) } \cong \bigvee_{i=1}^{2^{n+1}} P_i/\partial P_i,
	\end{equation}
	where the map in the middle is the canonical collapse map from the homotopy colimit to the strict colimit of the diagram, which is an equivalence because the diagram consists of a cell complex and a selection of subcomplexes that are closed under intersection. We are left with a sphere modulo all of the coordinate hyperplanes, which is homeomorphic to the wedge of all of the regions $P_i$ modulo their boundaries.
	
	It takes some effort to directly write down an apartment-like map for $P_i$: one has to take iterated mapping cylinder neighborhoods of each of the faces and use the cylinder coordinates to move along the homotopy colimit. It is much faster and easier to instead follow the strategy of \cref{st-apt-like}. We make a new definition of ``apartment-like'' for the pair $(S(\R^{n+1}),\bigcup_{V \in \cL_\circ \setminus \{\R^{n+1}\}} S(V))$, and we wait until we have mapped forward to this pair before we write down the apartment-like map explicitly.
	
	For each face $F$ of each of the polytopes $P_i$, we note that $F$ is a subset of $S^n$, and we say that a map $P_i \to S^n$ is apartment-like if its image is contained in $F \subseteq S^n$. Since each of these faces $F$ is contractible, the proof of \cref{apt-like-preserved-3} applies and shows that this space of apartment-like maps is weakly contractible. Furthermore any apartment-like map from $P_i$ to $\underset{V \in \cL}\hocolim\, S(V)$ by definition collapses down to an apartment-like map $P_i \to S^n$.
	
	Therefore, when evaluating the composite in \eqref{eq:pt-collapse} on the apartment class for $P_i$, we are allowed to go forward to the second quotient space and then choose the apartment-like map $P_i \to S^n$ directly. We choose it to be the identity of $P_i$, in other words the given inclusion $P_i \subseteq S^n$. Then after mapping to the last term of \eqref{eq:pt-collapse}, we get the identity map $P_i/\partial P_i = P_i/\partial P_i$. This verifies that the apartment classes define an isomorphism for $\cL_\circ$,
		\[ \begin{tikzcd} \Z^{2^{n+1}} \cong \Pt^{\cL_\circ}(S^n) \rar{\apt}[swap]{\cong} & \widetilde H_n(\PT^{\cL_\circ}(S^n)) \cong \Z^{2^{n+1}}. \end{tikzcd} \]
	
	This finishes the first step of the proof. For the second key inductive step, we assume the theorem is true for $\cL'$ and that $U \not\in \cL'$, and we form the diagram
	\begin{equation*}
	\begin{tikzcd}
		\underset{V \in \cL'}\hocolim\, S(V) \ar{rr}{\simeq} &&
		\underset{V \in (\cL' \Cup U)}\hocolim\, S(V) \\[-5pt]
		\underset{V \in \cL' \setminus \{\R^{n+1}\}}\hocolim\, S(V) \uar \rar{\simeq} &
		\underset{V \in (\cL' \Cup U) \setminus \{U,\R^{n+1}\}}\hocolim\, S(V) \rar &
		\underset{V \in (\cL' \Cup U) \setminus \{\R^{n+1}\}}\hocolim\, S(V) \uar \\[-5pt]
		&
		\underset{V \in (\cL' \cap U) \setminus \{U\}}\hocolim\, S(V) \rar \uar &
		\underset{V \in (\cL' \cap U)}\hocolim\, S(V). \uar
	\end{tikzcd}
	\end{equation*}
	To make the following argument less dense and more readable, we adopt the shorthand
	\[ h(\cL') \coloneq \hocolim_{V \in \cL'}\, S(V), \]
	so that the above diagram becomes the following illustration, which also shows which regions of the sphere contribute to each homotopy colimit.
	
	\begin{equation}\label{eqn:fact-pushout-sn}
	\centerline{
	\def\svgwidth{4.5in}
\begingroup%
  \makeatletter%
  \providecommand\color[2][]{%
    \errmessage{(Inkscape) Color is used for the text in Inkscape, but the package 'color.sty' is not loaded}%
    \renewcommand\color[2][]{}%
  }%
  \providecommand\transparent[1]{%
    \errmessage{(Inkscape) Transparency is used (non-zero) for the text in Inkscape, but the package 'transparent.sty' is not loaded}%
    \renewcommand\transparent[1]{}%
  }%
  \providecommand\rotatebox[2]{#2}%
  \newcommand*\fsize{\dimexpr\f@size pt\relax}%
  \newcommand*\lineheight[1]{\fontsize{\fsize}{#1\fsize}\selectfont}%
  \ifx\svgwidth\undefined%
    \setlength{\unitlength}{320.84617396bp}%
    \ifx\svgscale\undefined%
      \relax%
    \else%
      \setlength{\unitlength}{\unitlength * \real{\svgscale}}%
    \fi%
  \else%
    \setlength{\unitlength}{\svgwidth}%
  \fi%
  \global\let\svgwidth\undefined%
  \global\let\svgscale\undefined%
  \makeatother%
  \begin{picture}(1,1.0232106)%
    \lineheight{1}%
    \setlength\tabcolsep{0pt}%
    \put(0,0){\includegraphics[width=\unitlength,page=1]{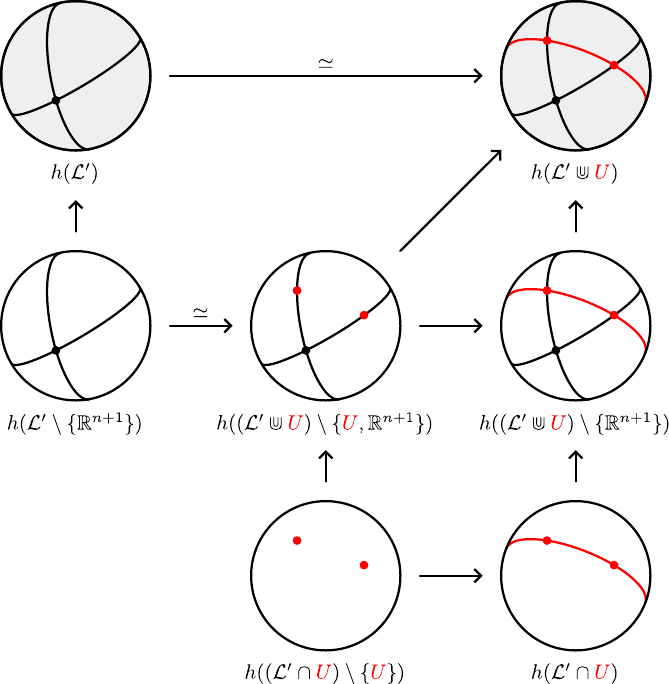}}%
  \end{picture}%
\endgroup%

	}
	\end{equation}

	The maps marked $\simeq$ are equivalences because the inclusions
	\[ \cL' \to (\cL' \Cup U), \qquad \cL' \setminus \{\R^{n+1}\} \to (\cL' \Cup U) \setminus \{U,\R^{n+1}\} \]
	are homotopy terminal (i.e.\ homotopy colimit preserving). In each case, this means that for each space $V$ in the larger category, the poset of all $V'$ in the smaller category that contain $V$ form a contractible poset---indeed, it has an initial object, given by intersecting all such $V'$ together. (This is also how we got the deformation retract in the proof of \cref{main-en}.)
	
	The remaining square of \eqref{eqn:fact-pushout-sn} is a pushout because, as before, a space in $(\cL' \Cup U) \setminus \{\R^{n+1}\}$ either has to be a subspace of $U$, or an element of $(\cL' \Cup U) \setminus \{U,\R^{n+1}\}$, or both, which happens precisely when it is an element of $\cL' \cap U \setminus \{U\}$. Note that by \cref{pt-subcomplex}, all of these spaces are subcomplexes of the homotopy colimit $h((\cL' \Cup U) \setminus \{\R^{n+1}\})$, so the pushout is a homotopy pushout.
	
	Now we take the triangle in \eqref{eqn:fact-pushout-sn}, write it as nested subspaces $A \subseteq B \subseteq C$, and form the corresponding cofiber sequence of the form $B/A \to C/A \to C/B$. By the above equivalences and homotopy pushout, this cofiber sequence can be rewritten as
	\begin{equation}\label{eqn:cofiber-seq-sn}
	\begin{tikzcd}[row sep = 1em]
		\frac{ \displaystyle h(\cL' \cap U) }{ \displaystyle h((\cL' \cap U) \setminus \{U\}) } \rar &
		\frac{ \displaystyle h(\cL') }{ \displaystyle h(\cL' \setminus \{\R^{n+1}\}) } \rar &
		\frac{ \displaystyle h(\cL' \Cup U) }{ \displaystyle h((\cL' \Cup U) \setminus \{\R^{n+1}\}) } \\
		\PT^{\cL' \cap U}(S(U)) \uar[equal] \rar &
		\PT^{\cL'}(S^n) \uar[equal] \rar  &
		\PT^{\cL' \Cup U}(S^n). \uar[equal]
	\end{tikzcd}
	\end{equation}
	By inductive hypothesis the first term is a wedge of $(n-1)$-spheres, and by assumption the second term is a wedge of $n$-spheres. Therefore the map between them is nullhomotopic, so the cofiber is a wedge of $n$-spheres, proving the first half of the theorem for $\cL' \Cup U$.
	
	We then form the map of exact sequences
	\begin{equation}\label{st_sn_map_of_les}
	\begin{tikzcd}[column sep = 1.5em] 0 \rar & \Pt^{\cL'}(S^n) \rar \dar{\apt}[swap]{\cong} & \Pt^{\cL' \Cup U}(S^n) \dar{\apt} \rar & \Pt^{\cL' \cap U}(S(U)) \dar[dashed] \rar & 0 \\[-5pt]
		0 \rar & \widetilde{H}_n(\PT^{\cL'}(S^n)) \rar &  \widetilde{H}_n(\PT^{\cL' \Cup U}(S^n)) \rar & \widetilde{H}_{n-1}(\PT^{\cL' \cap U}(S(U))) \rar & 0.\end{tikzcd}
	\end{equation}
	 As before, we need to show that the dashed map is given by the apartment classes in $U$. We need to redo the argument in this case, since we are working with $\PT$ rather than $\ST$.

	We think of apartment maps as maps of pairs $(Q,\partial Q) \to (h(\cL),h(\cL \setminus \{\R^{n+1}\}))$. The connecting homomorphism for the cofiber sequence \eqref{eqn:cofiber-seq-sn} becomes
	\begin{equation}\label{snake_map_expanded_sn}
		\begin{aligned}
			& \ H_n(h(\cL' \Cup U),h((\cL' \Cup U) \setminus \{\R^{n+1}\})) \\
			\overset\partial\longrightarrow & \ H_{n-1}(h((\cL' \Cup U) \setminus \{\R^{n+1}\}),\varnothing) \\
			\longrightarrow & \ H_{n-1}(h((\cL' \Cup U) \setminus \{\R^{n+1}\}),h((\cL' \Cup U) \setminus \{U,\R^{n+1}\})) \\
			\cong & \ H_{n-1}(h(\cL' \cap U),h((\cL' \cap U) \setminus \{U\})),
		\end{aligned}
	\end{equation}
	where the first map takes a relative chain in $H_n(C,B)$ to its boundary in $H_{n-1}(B)$, the second map projects this to a relative chain in $H_{n-1}(B,A)$, and the last map uses excision along the pushout square in \eqref{eqn:fact-pushout-sn} to rewrite the result.
	
	Now we compute the dashed map of \eqref{st_sn_map_of_les} on an element $[P] \in \Pt^{\cL' \cap U}(S(U))$ where $P$ is strongly convex. It lifts to $[Q] \in \Pt^{\cL' \Cup U}(S^n)$, where $Q \subseteq U_+$ is a strongly convex $\cL' \Cup U$-polytope with $P$ as a facet by \cref{new_generators_sn}. We then take the apartment class for $Q$, whose boundary fits into the following commuting diagram of pairs:
	\[\begin{tikzcd}
		(\partial Q, \varnothing) \dar \rar{\apt} & (h((\cL' \Cup U) \setminus \{\R^{n+1}\}),\varnothing) \dar \\[-5pt]
		(\partial Q, \partial Q \setminus \textup{Int }P) \rar{\apt} & (h((\cL' \Cup U) \setminus \{\R^{n+1}\}),h((\cL' \Cup U) \setminus \{U,\R^{n+1}\})) \\[-5pt]
		(P,\partial P) \uar{\cong_{H_*}} \rar{\apt} & (h((\cL' \cap U)),h((\cL' \cap U) \setminus \{U\})). \uar[swap]{\cong_{H_*}}
	\end{tikzcd}\]
	where the maps marked $\cong_{H_*}$ induce isomorphisms on homology. We conclude that the apartment class of $Q$ under \eqref{snake_map_expanded_sn} goes to the apartment class of $P$, as desired.
	 
	 Since the dashed map of \eqref{st_sn_map_of_les} is given by apartment classes, it is an isomorphism by inductive hypothesis. Therefore the middle vertical map is an isomorphism, and we conclude that $\cL' \Cup U \in \mathcal{P}_\cL$. This finishes the second step of the proof.
	
	The proof of the third step (the one that takes the colimit $\cup_\lambda \cL_\lambda$) is identical to the proof in \cref{main-en}, except that we take the colimit of the complexes $\PT^{\cL_\lambda}(S^n)$, and the result is an $(n-1)$-connected cell complex that is $n$-dimensional by \cref{pt-subcomplex}, and is therefore a wedge of $n$-spheres.
	\end{proof}

	\subsection{The non-admissible case}
	
	If $\cL$ is not admissible, then the intersection of all the hyperplanes in $\cL$ is some non-trivial subspace $U \subseteq \R^{n+1}$. We assume throughout this section that $\cL^{n-1}$ is non-empty, so that $U$ is a proper subset of $\R^{n+1}$ and hence $\dim U^\perp > 0$. It follows that each hyperplane in $\cL$ is of the form $U \oplus H$ for a unique hyperplane $H \subseteq U^\perp$. We therefore get a smaller system of hyperplanes inside the vector space $U^\perp$, denoted $\cL \cap U^\perp$. As the intersection of all the hyperplanes in $\cL \cap U^\perp$ is $U \cap U^\perp = \{0\}$, we get:
	
	\begin{lemma}
		$\cL \cap U^\perp$ is admissible.
	\end{lemma}
		
	\vspace{1em}
	\centerline{
	\def\svgwidth{4.5in}
\begingroup%
  \makeatletter%
  \providecommand\color[2][]{%
    \errmessage{(Inkscape) Color is used for the text in Inkscape, but the package 'color.sty' is not loaded}%
    \renewcommand\color[2][]{}%
  }%
  \providecommand\transparent[1]{%
    \errmessage{(Inkscape) Transparency is used (non-zero) for the text in Inkscape, but the package 'transparent.sty' is not loaded}%
    \renewcommand\transparent[1]{}%
  }%
  \providecommand\rotatebox[2]{#2}%
  \newcommand*\fsize{\dimexpr\f@size pt\relax}%
  \newcommand*\lineheight[1]{\fontsize{\fsize}{#1\fsize}\selectfont}%
  \ifx\svgwidth\undefined%
    \setlength{\unitlength}{329.33156938bp}%
    \ifx\svgscale\undefined%
      \relax%
    \else%
      \setlength{\unitlength}{\unitlength * \real{\svgscale}}%
    \fi%
  \else%
    \setlength{\unitlength}{\svgwidth}%
  \fi%
  \global\let\svgwidth\undefined%
  \global\let\svgscale\undefined%
  \makeatother%
  \begin{picture}(1,0.35672467)%
    \lineheight{1}%
    \setlength\tabcolsep{0pt}%
    \put(0,0){\includegraphics[width=\unitlength,page=1]{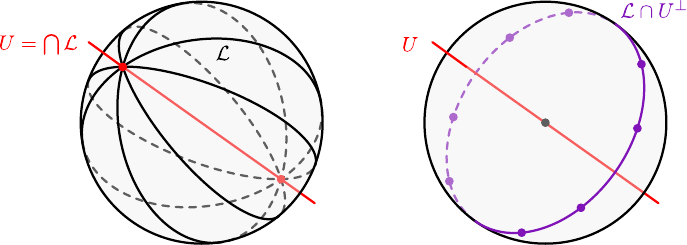}}%
  \end{picture}%
\endgroup%

	}
	\vspace{1em}
	
	\begin{lemma}\label{pt-suspended-0}
		In this case, intersecting each $\cL$-polytope with $S(U^\perp)$ defines an isomorphism 
		\[ \Pt^\cL(S^n) \cong \Pt^{\cL \cap U^\perp}(S(U^\perp)). \]
	\end{lemma}
	
	\begin{proof}
		Every convex $\cL$-polytope is an intersection of half-spaces of the form $U \oplus V^+$, where $V^+$ is a half-space in $U^\perp$. Therefore intersecting with $U$ is invertible, with inverse given by adding $U$ back in.
	\end{proof}

	Geometrically, the inverse of this isomorphism takes the \emph{join} of each polytope $P \subseteq S(U^\perp)$ with the sphere $S(U)$. It is convenient to use the following model, homeomorphic to the usual join construction:
	
	\begin{definition}\label{df:join}
		Given subsets $K \subseteq S(U)$ and $L \subseteq S(U^\perp)$, their \emph{join} is the subset $K * L \subseteq S^n$ formed as the image of the map
		\begin{align*}
			K \times I \times L &\longrightarrow S^n = (\R^{n+1} \setminus \{0\})/\R_+ \\
				(k,t,\ell) &\mapsto (1-t)k + t\ell.
		 \end{align*}
	\end{definition}
	
	For example, when $P$ is a polytope in $S(U^\perp)$, the join $S(U) * P$ is a polytope in $S^n$. Every $\cL$-polytope is of this form, because it is obtained by adding to $P$ all possible vectors in $U$ and rescaling to land on the sphere, which lines up with \cref{df:join}.
	
	\vspace{1em}
	\centerline{
	\def\svgwidth{1.8in}
\begingroup%
  \makeatletter%
  \providecommand\color[2][]{%
    \errmessage{(Inkscape) Color is used for the text in Inkscape, but the package 'color.sty' is not loaded}%
    \renewcommand\color[2][]{}%
  }%
  \providecommand\transparent[1]{%
    \errmessage{(Inkscape) Transparency is used (non-zero) for the text in Inkscape, but the package 'transparent.sty' is not loaded}%
    \renewcommand\transparent[1]{}%
  }%
  \providecommand\rotatebox[2]{#2}%
  \newcommand*\fsize{\dimexpr\f@size pt\relax}%
  \newcommand*\lineheight[1]{\fontsize{\fsize}{#1\fsize}\selectfont}%
  \ifx\svgwidth\undefined%
    \setlength{\unitlength}{138.38223864bp}%
    \ifx\svgscale\undefined%
      \relax%
    \else%
      \setlength{\unitlength}{\unitlength * \real{\svgscale}}%
    \fi%
  \else%
    \setlength{\unitlength}{\svgwidth}%
  \fi%
  \global\let\svgwidth\undefined%
  \global\let\svgscale\undefined%
  \makeatother%
  \begin{picture}(1,0.85788989)%
    \lineheight{1}%
    \setlength\tabcolsep{0pt}%
    \put(0,0){\includegraphics[width=\unitlength,page=1]{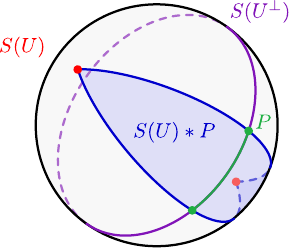}}%
  \end{picture}%
\endgroup%

	}
	\vspace{1em}
	
	Let $S^U \coloneq D(U)/S(U)$, homeomorphic to the one-point compactification of $U$. The following is a direct consequence of the usual join construction:

	\begin{lemma}\label{quotient-of-joins}
		If $\dim U^\perp > 0$, then for any pair of subspaces $A \subseteq A'$, there is a homeomorphism
		\[  \frac{ S(U) * A' }{ S(U) * A } \cong S^U \sma (A'/A) = \Sigma^U(A'/A) \]
	    natural in inclusions of pairs.
	\end{lemma}
	
	\begin{proposition}\label{pt-suspended}
		If $U = \bigcap \cL$ and $\dim U^\perp > 0$, then there is a homeomorphism
		\[ \PT^{\cL}(S^n) \cong \Sigma^U \PT^{\cL \cap U^\perp}(S(U^\perp)). \]
	\end{proposition}
	
	\begin{proof}
		Since $\cL$ consists entirely of subspaces of the form $U \oplus V$ for $V \in \cL \cap U^\perp$, we can rewrite the definition of $\PT^\cL(S^n)$ from \cref{df:st-and-pt} as
	\begin{align*}
		\PT^{\cL}(S^n) &= \frac{ \underset{V \in (\cL \cap U^\perp)}\hocolim\, S(U) * S(V) }{ \underset{V \in (\cL \cap U^\perp) \setminus \{U^\perp\}}\hocolim\, S(U) * S(V) }.
	\end{align*}
	We write the numerator and denominator as pushouts of the following two diagrams, using the fact that the homotopy colimit commutes with products and pushouts:
	\[ \begin{tikzcd}[column sep = 2.5em, row sep = .5em]
		S(U) \times \underset{V \in (\cL \cap U^\perp)}\hocolim\, {*} & \lar S(U) \times \underset{V \in (\cL \cap U^\perp)}\hocolim\, S(V) \rar & {*} \times \underset{V \in (\cL \cap U^\perp)}\hocolim\, S(V) \\
		S(U) \times \underset{V \in (\cL \cap U^\perp) \setminus \{U^\perp\}}\hocolim\, {*} & \lar S(U) \times \underset{V \in (\cL \cap U^\perp) \setminus \{U^\perp\}}\hocolim\, S(V) \rar & {*} \times \underset{V \in (\cL \cap U^\perp) \setminus \{U^\perp\}}\hocolim\, S(V)
	\end{tikzcd} \]
	Since the posets $(\cL \cap U^\perp)$ and $(\cL \cap U^\perp) \setminus \{U^\perp\}$ both have an initial object, the homotopy colimit of a constant diagram on either one of these is equivalent to the value of the constant functor. Therefore, the homotopy colimits appearing in the left terms are contractible and each of the two pushouts simplifies to the join of $S(U)$ with the homotopy colimit appearing in the right term. Hence, we get an equivalence
	\begin{align*}
		\PT^{\cL}(S^n) &\simeq \frac{ S(U) * \underset{V \in (\cL \cap U^\perp)}\hocolim\, S(V) }{ S(U) * \underset{V \in (\cL \cap U^\perp) \setminus \{U^\perp\}}\hocolim\, S(V) }.
	\end{align*}
	Using \cref{quotient-of-joins}, this simplifies to
	\[
		\Sigma^U \left( \frac{ \underset{V \in (\cL \cap U^\perp)}\hocolim\, S(V) }{ \underset{V \in (\cL \cap U^\perp) \setminus \{U^\perp\}}\hocolim\, S(V) } \right)
		\cong \Sigma^U \PT^{\cL \cap U^\perp}(S(U^\perp)).\qedhere
	\]
	\end{proof}
	
	Given all of this, we clearly ought to define an apartment-like map on $S(U) * P$ to be an apartment-like map on $P$, joined with the identity map of $S(U)$. When we take the quotient by the boundary, this gives the $U$-suspension of the suspended apartment class for $P$.
	
	\begin{definition}\label{apartment_pt_sn_2}
		If $P$ is a strongly convex polytope in $S(U^\perp)$, we define the \emph{suspended apartment} of $S(U) * P$ to be the operation $\Sigma^U(-)$ applied to the apartment map of $P$ in $\PT^{\cL \cap U^\perp}(S(U^\perp))$:
		\[ \begin{tikzcd}[column sep = 4em] \Sigma^U (P/\partial P) \rar{\Sigma^U(\apt)} & \Sigma^U \PT^{\cL \cap U^\perp}(S(U^\perp)) \cong \PT^{\cL}(S^n). \end{tikzcd} \]
	\end{definition}
	
	\begin{corollary}[Solomon-Tits for $\cL$-polytopes in $S^n$, non-empty case]
		For all $n \geq 0$, if $\cL$ is generated by hyperplanes and $\cL^{n-1} \neq \varnothing$, then $\PT^\cL(S^n)$ is equivalent to a wedge of $n$-spheres and the apartment classes of \cref{apartment_pt_sn_2} induce an isomorphism
		\[ \Pt^\cL(S^n) \overset{\cong}\longrightarrow \widetilde H_{n}(\PT^\cL(S^n)). \]
	\end{corollary}
	\begin{proof}
		By \cref{main-sn}, the space $\PT^{\cL \cap U^\perp}(S(U^\perp))$ is a wedge of spheres of dimension $\dim U^\perp$. Since $\cL^{n-1}$ is non-empty, $\dim U^\perp > 0$. Therefore by \cref{pt-suspended} the space $\PT^\cL(S^n)$ is a wedge of spheres of dimension $\dim U + \dim U^\perp = n$. Furthermore, since suspension by $U$ induces an isomorphism on homology, the apartment maps as defined in \cref{apartment_pt_sn_2} induce an isomorphism
		\[ \begin{tikzcd}[column sep = 3em] \Pt^\cL(S^n) \cong \Pt^{\cL \cap U^\perp}(S(U^\perp)) \rar{\apt}[swap]{\cong} & H_n(\Sigma^U \PT^{\cL \cap U^\perp}(S(U^\perp))) \cong H_n(\PT^\cL(S^n)) \end{tikzcd} \]
		with the first isomorphism coming from \cref{pt-suspended-0}.
	\end{proof}
	
	\begin{remark}It turns out that a version of this theorem is true if $\cL^{n-1}$ is empty as well, in other words if $\cL = \{\R^{n+1}\}$. In that case the only $\cL$-polytope is the whole sphere $S^n$, so we would need $\PT^\varnothing(S^n) \simeq S^n$. This is not true in the definition that we gave in \cref{df:st-and-pt}, but there is an alternate model for the one-fold suspension $\Sigma\PT$ that extends to the case where $\cL = \varnothing$, and in that case gives the expected answer of $\Sigma\PT \simeq S^{n+1}$:
	\[ \Sigma\PT^\cL(S(W)) \coloneq \frac{ \underset{V \in \cL}\hocolim\, S^V }{ \underset{V \in \cL \setminus \{W\}}\hocolim\, S^V }. \]
	The details of this construction will appear in \cite{KKMMW2}.\end{remark}

	\section{Generation by both points and hyperplanes}
	
	We say that $\cL$ is \emph{generated by points and hyperplanes} if it is both generated by points and generated by hyperplanes. In this stronger but still frequently-occurring case, we can relate the polytope and Lee--Szczarba groups (\cref{lem:gen-points-pt-to-ls}). 
	
	We assume that $n \geq 1$ throughout, so that the dimension of a point is less than or equal to the dimension of a hyperplane. Note that in the Euclidean and spherical case a collection $\cL$ that is generated by points and hyperplanes must be admissible: there is no way to be generated by points unless $\cL^0$ has more than one 0-dimensional subspace (since $n \geq 1$).
	
	\begin{lemma}\label{cut_into_simplices}
		If $\cL$ is generated by points and hyperplanes, then every convex $\cL$-polytope can be subdivided into $\cL$-simplices.
	\end{lemma}

	\begin{proof}
	Since $\cL$ has to be admissible in the spherical case, we may assume that we have cut the $\cL$-polytope into strongly convex pieces and focus on one such piece in this case. In the following, we consider a (strongly) convex $\cL$-polytope with vertices $\{v_1, \dots, v_k\}$ in $X^n$.
		
	We start with the Euclidean case and follow the iterated placing algorithm from \cite[16.2.1(2)]{lee_santos}: choose a total ordering on the vertices $\{v_1,\ldots,v_k\}$, and then inductively define a triangulation of the convex hull of $\{v_1,\ldots,v_i\}$ as follows. Given a triangulation of $\{v_1,\ldots,v_{i-1}\}$, if $v_i$ is outside the affine span of $\{v_1,\ldots,v_{i-1}\}$, then we ``cone off'' to $v_i$ by taking the join of each of the simplices in the given triangulation with $v_i$.
	
	\vspace{1em}
	\centerline{
	\def\svgwidth{1.45in}
\begingroup%
  \makeatletter%
  \providecommand\color[2][]{%
    \errmessage{(Inkscape) Color is used for the text in Inkscape, but the package 'color.sty' is not loaded}%
    \renewcommand\color[2][]{}%
  }%
  \providecommand\transparent[1]{%
    \errmessage{(Inkscape) Transparency is used (non-zero) for the text in Inkscape, but the package 'transparent.sty' is not loaded}%
    \renewcommand\transparent[1]{}%
  }%
  \providecommand\rotatebox[2]{#2}%
  \newcommand*\fsize{\dimexpr\f@size pt\relax}%
  \newcommand*\lineheight[1]{\fontsize{\fsize}{#1\fsize}\selectfont}%
  \ifx\svgwidth\undefined%
    \setlength{\unitlength}{93.92967742bp}%
    \ifx\svgscale\undefined%
      \relax%
    \else%
      \setlength{\unitlength}{\unitlength * \real{\svgscale}}%
    \fi%
  \else%
    \setlength{\unitlength}{\svgwidth}%
  \fi%
  \global\let\svgwidth\undefined%
  \global\let\svgscale\undefined%
  \makeatother%
  \begin{picture}(1,0.79968899)%
    \lineheight{1}%
    \setlength\tabcolsep{0pt}%
    \put(0,0){\includegraphics[width=\unitlength,page=1]{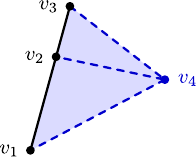}}%
  \end{picture}%
\endgroup%

	}
	\vspace{1em}
	
	If $v_i$ is inside the affine span of $\{v_1,\ldots,v_{i-1}\}$, then we take only the ``visible facets'' and take their join with $v_i$. (A facet $F$ is visible to $v_i$ if the point $v_i$ and the convex hull of $\{v_1,\ldots,v_{i-1}\}$ lie on opposite sides of $\spa F$.)
	
	\vspace{1em}
	\centerline{
	\def\svgwidth{1.7in}
\begingroup%
  \makeatletter%
  \providecommand\color[2][]{%
    \errmessage{(Inkscape) Color is used for the text in Inkscape, but the package 'color.sty' is not loaded}%
    \renewcommand\color[2][]{}%
  }%
  \providecommand\transparent[1]{%
    \errmessage{(Inkscape) Transparency is used (non-zero) for the text in Inkscape, but the package 'transparent.sty' is not loaded}%
    \renewcommand\transparent[1]{}%
  }%
  \providecommand\rotatebox[2]{#2}%
  \newcommand*\fsize{\dimexpr\f@size pt\relax}%
  \newcommand*\lineheight[1]{\fontsize{\fsize}{#1\fsize}\selectfont}%
  \ifx\svgwidth\undefined%
    \setlength{\unitlength}{107.91535117bp}%
    \ifx\svgscale\undefined%
      \relax%
    \else%
      \setlength{\unitlength}{\unitlength * \real{\svgscale}}%
    \fi%
  \else%
    \setlength{\unitlength}{\svgwidth}%
  \fi%
  \global\let\svgwidth\undefined%
  \global\let\svgscale\undefined%
  \makeatother%
  \begin{picture}(1,0.84563331)%
    \lineheight{1}%
    \setlength\tabcolsep{0pt}%
    \put(0,0){\includegraphics[width=\unitlength,page=1]{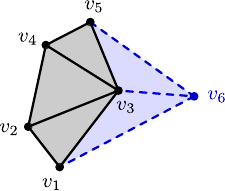}}%
  \end{picture}%
\endgroup%

	}
	\vspace{1em}
	
	This algorithm decomposes the polytope into simplices, each of which is a convex $\cL$-polytope because its vertices are in $\cL^0$. 
	
	To finish, we observe that this algorithm works in all three geometries: given its validity in Euclidean geometry, we deduce its validity in hyperbolic geometry or in an open hemisphere in spherical geometry by radially projecting from the hyperboloid $-x_0^2+\sum_{i=1}^n x_i^2 = -1$ or the hemisphere $\sum_{i=0}^n x_i^2 = 1, x_0 > 0$ to the hyperplane $x_0 = 0$. This sends the geometric subspaces to affine-linear subspaces, so that if we carry out this algorithm in Euclidean space $E^n$, we get a resulting triangulation in $H^n$ or $S^n$ as well.
	\end{proof}
	
	\begin{lemma}\label{lem:open_hemisphere}
		For any set of $\leq (n+1)$ points in $S^n$, if they are linearly independent as vectors in $\R^{n+1}$, then they lie in an open hemisphere. Any finite collection of points can be made to lie in an open hemisphere by negating some of them.
	\end{lemma}
	
	\begin{proof}
		We prove the first statement, the proof of the second is similar and easier. Without loss of generality we have $(n+1)$ points, and inductively, the first $n$ points lie in an open hemisphere of an $S^{n-1}$ inside $S^n$. Changing coordinates, we have $S^n = S^{n-2} * S^1$, with the first $n$ points all lying in the join of $S^{n-2}$ with a single point $x \in S^1$, with positive join coordinate (so not in $S^{n-2}$). The remaining point also has positive join coordinate with some point $y \in S^1$ with $y \neq -x$. Now, $x,y \in S^1$ must lie in an open hemisphere of $S^1$. Taking the join of this with $S^{n-2}$ gives the open hemisphere in $S^n$ containing all $(n+1)$ of our points.
	\end{proof}

	\begin{lemma}\label{lem:map-pt-to-ls}
		If $\cL$ is generated by points and hyperplanes, there is a unique map
		\[ \Pt^\cL(X^n) \longrightarrow \Ls^\cL(X^n) \]
		of $\Z[G_\cL]$-modules with the property that every simplex is sent to its tuple of vertices, with ordering chosen to agree with some fixed orientation on $X^n$.
	\end{lemma}

	\begin{proof}
		Recall from \cref{st-apt-like-2} that a map $(P,\partial P) \to (\CT^\cL(X^n),\T^\cL(X^n))$ is apartment-like if each face $F$ of $\partial P$ is sent into the contractible complex $\CT^{\cL}(\spa F)$. By \cref{apt-like-preserved-2}, the space of such is weakly contractible and therefore non-empty. This is used to construct the left-hand diagonal map in the following diagram.
		\[\begin{tikzcd} \Pt^\cL(X^n) \arrow{rd}[swap]{\apt}  \arrow[dashed]{rr}& & \Ls^\cL(X^n) \arrow{ld}{\apt}[swap]{\cong} \\[-5pt]
			& \widetilde{H}_n(\ST^\cL(X^n)), &  \end{tikzcd} \]
		The right-hand diagonal map is an isomorphism by \cref{thm:generated-by-points}, so there is a unique dashed map making the diagram commute. By \cref{cut_into_simplices} each $\cL$-polytope can be cut into simplices, and if $P$ is a simplex on vertices $(x_0,\ldots,x_n)$ (with ordering chosen to agree with the orientation of $X^n$) then $\apt([P]) = \apt(x_0,\ldots,x_n)$. So the dashed map is the only map that exists with the required behaviour on simplices.
	\end{proof}

	\begin{theorem}\label{lem:gen-points-pt-to-ls}
		For any collection $\cL$ that is generated by points and hyperplanes:
		\begin{enumerate}
			\item The canonical map $\Pt^\cL(E^n) \to \Ls^\cL(E^n)$ is an isomorphism.
			\item The canonical map $\Pt^\cL(H^n) \to \Ls^\cL(H^n)$ is an isomorphism.
			\item The canonical map $\Pt^\cL(S^n) \to \Ls^\cL(S^n)$ is surjective, and the kernel is generated by joins of the form $\{x,-x\} * P$, where $\{x,-x\} \in \cL^0$ and $P$ is an $(n-1)$-simplex with vertices in $\cL^0$.
		\end{enumerate}
	\end{theorem}
	
	Note that in the case that $\cL$ is the set of all subspaces, this theorem appears in the literature in e.g.~\cite[Theorem 2.10 and Theorem 3.13]{dupont_book}, \cite[Section 5]{dupont_82}, \cite[Section 3]{morelli_93}.

	\begin{proof}
		 In the geometries $E^n$ and $H^n$, we define the inverse $\phi\colon \Ls^\cL(X^n) \to \Pt^\cL(X^n)$ by sending each $(n+1)$-tuple of points $(x_0,\ldots,x_n)$ to the unique simplex that has those vertices, $\pm[x_0,\ldots,x_n]$. The sign records whether the ordering on the vertices agrees with our fixed orientation on $X^n$. To see that $\phi$ is well-defined, we check that each expression of the form
		\begin{equation}\label{eq:simplicial-relation}
		\sum_{i=0}^{n+2} (-1)^i (y_0,\ldots,\hat{y_i},\ldots,y_{n+1})
		\end{equation}
		is sent to zero in $\Pt^\cL(X^n)$. This can be proven as in \cite[Theorem 2.10]{dupont_book}, but we give some details. We take the image of \eqref{eq:simplicial-relation} in $\Pt^\cL(X^n)$,
		\begin{equation*}\label{eq:simplicial-relation-2}
		\sum_{i=0}^{n+2} \pm [y_0,\ldots,\hat{y_i},\ldots,y_{n+1}],
		\end{equation*}
		and write it as the image of a map $f\colon \partial \Delta^{n+1} \to X^n$, with signs recording whether each face of $\partial \Delta^{n+1}$ is sent to $X^n$ in a way that preserves or reverses the orientation coming from $\partial \Delta^{n+1} \cong S^n$. By further subdividing, we can express this element of $\Pt^\cL(X^n)$ in the form $\sum_i n_i [P_i]$, where the $P_i$ have disjoint interiors. However, the map $f$ must be degree zero since $X^n$ is contractible, and therefore the local degree formula for homology gives $n_i = 0$ for all $i$. We conclude that the image of \eqref{eq:simplicial-relation} in $\Pt^\cL(X^n)$ must be zero.
		
		Now that $\phi\colon \Ls^\cL(X^n) \to \Pt^\cL(X^n)$ is well-defined, it is clearly a right inverse of the canonical map from \cref{lem:map-pt-to-ls}, and is therefore injective. By \cref{cut_into_simplices}, $\phi$ is surjective as well, and therefore both $\phi$ and the canonical map are bijections.
		
		\medskip

		 In the geometry $S^n$, the canonical map is surjective because each non-zero element of $\Ls^\cL(S^n)$ is given by a tuple $(x_0,\ldots,x_n)$ of linearly independent vectors, which by \cref{lem:open_hemisphere} lie in an open hemisphere and define a strongly convex $\cL^0$-polytope in $\Pt^\cL(S^n)$. If we let $J^\cL$ be the subgroup of $\Pt^\cL(S^n)$ generated by joins of the form $\{x,-x\} * P$, then the composite
		 \[ J^\cL \longrightarrow \Pt^\cL(S^n) \longrightarrow \Ls^\cL(S^n) \]
		 vanishes because the two joins $\{x\} * P$ and $\{-x\} * P$ go to the same element of $\Ls^\cL(S^n)$ by \cref{rem:lee-szczarba-group}, but with opposite signs. The canonical map thus factors through the quotient,
		 \[ \Pt^\cL(S^n)/J^\cL \longrightarrow \Ls^\cL(S^n). \]
		 
		 To define an inverse map $\phi\colon \Ls^\cL(S^n) \to \Pt^\cL(S^n)/J^\cL$, we send each tuple of linearly independent length-one vectors $(x_0,\ldots,x_n)$ to the simplex $\pm[\pm x_0,\ldots,\pm x_n]$, with sign on the outside chosen to reflect whether the simplex on the inside respects the fixed orientation on $S^n$. There are $2^{n+1}$ different possibilities here, but they all give the same element of $\Pt^\cL(S^n)/J^\cL$, since flipping the sign on any point $x_i$ gives two different simplices with opposite signs, so their difference has the same sign, and is therefore a join of the form $\{x_i,-x_i\} * [\pm x_1, \ldots, \hat{\pm x_i}, \ldots, \pm x_n]$.
		 
		 To finish verifying $\phi$ is well-defined, we have to check it sends the relations \eqref{eq:simplicial-relation} to zero. By \cref{rem:lee-szczarba-group}, we can change the presentation of $\Ls^\cL(S^n)$ by first imposing the relations
			\[\qquad \qquad (x_0,\ldots,x_{i-1},x_i,x_{i+1},\ldots,x_n) = (x_0,\ldots,x_{i-1},-x_i,x_{i+1},\ldots,x_n), \]
		and then imposing the simplicial relation \eqref{eq:simplicial-relation}. We have already checked that $\phi$ is well-defined after this first set of relations. But now, when checking the simplicial relations, we can take the terms $y_i$ and negate whichever ones we want before checking the relation. By \cref{lem:open_hemisphere}, we can therefore without loss of generality assume that the points $(y_0,\ldots,y_n)$ lie in an open hemisphere. From here, the argument that \eqref{eq:simplicial-relation} goes to zero is the same as in the Euclidean and hyperbolic case.
		
		Finally, now that $\phi\colon \Ls^\cL(S^n) \to \Pt^\cL(S^n)/J^\cL$ is well-defined, as before we argue it is surjective and a right inverse of the canonical map, and is therefore an isomorphism.
	\end{proof}
	
	\begin{remark}
		\label{rem:generation-by-points-and-hyperplanes-eulcidean-and-hyperbolic}
		In the Euclidean and hyperbolic cases, \cref{lem:gen-points-pt-to-ls} does not assume that the apartment class homomorphism $\Pt^\cL(X^n) \to \widetilde{H}_n(\PT^\cL(X^n))$ is an isomorphism, but we get to conclude that it is an isomorphism from the proof. It therefore gives the conclusion of \cref{thm:generated-by-hyperplanes} under different hypotheses, namely when $\cL$ is generated by both points and hyperplanes. In the Euclidean case, the hypotheses imply admissibility and hence the result is covered by the previous version (\cref{main-en}). In the hyperbolic case, however, it avoids the admissibility condition in \cref{admissible-hn} involving the nested polytopes $A_i$, and therefore covers some cases not covered by the previous version (\cref{main-hn}).
	\end{remark}
	
	\begin{remark}
		In the spherical case, when $\cL$ is closed under taking orthogonal complements, every polytope of the form $\{x,-x\} * P$ can be rewritten as $\{x,-x\} * P'$, where $P'$ is a polytope in the orthogonal complement of the span of $x$. (We see this by comparing their cones in $\R^{n+1}$: we obtain the cone for $P'$ by taking the orthogonal projection of the cone for $P$.) In this case we get a simpler description of the kernel where we only need to take orthogonal subspaces, as in \cite[Theorem 3.13]{dupont_book}.
	\end{remark}

\section{A conjectural resolution of the polytope group}

\cref{thm:generated-by-points} states that if $\cL$ is generated by points, then \[\Ls^\cL(X^n) \cong \widetilde{H}_n(\ST^\cL(X^n)).\] We defined $\Ls^\cL(X^n)$ via a presentation (see \cref{def:lee-szczarba-group}). In this section, we explain this presentation extends to a resolution and then suggest an analogue for the polytope group.

We proved
\[\frac{\Tpl(\cL^0)}{\Tpl(\cL^0)^{n-1}} \simeq \ST^\cL(X^n)\]  if $\cL$ is generated by points and hence the relative cellular chain complex for the quotient space $\Tpl(\cL^0) / \Tpl(\cL^0)^{n-1}$ gives a resolution of $\widetilde{H}_n(\ST^\cL(X^n))$. The $p$-syzygies of this resolution are given by the relative $p+n$ chains. These chains can be described as the free abelian group on the $(n+p)$-tuples of elements of $\cL^0$ modulo all tuples lying in an $(n-1)$-dimensional subspace. The differential is given by the alternating sum of forgetting elements of the tuple. See \cite[Theorem 3.1]{LeeSzczarba} for a similar argument.

\cref{thm:generated-by-hyperplanes} states that if $\cL$ is generated by hyperplanes and is admissible, then \[\Pt^\cL(X^n)\cong \widetilde{H}_n(\PT^\cL(X^n)).\] The restricted polytope group is defined in \cref{def:polytope-group} via the following presentation. $\Pt^\cL(X^n)$ is the free abelian group on the $\cL$-polytopes $P \subseteq X^n$, modulo the relation that whenever $P$ subdivides into finitely many $\cL$-polytopes $\{P_i\}$, we have $[P] = \sum_i [P_i]$. We will give a conjectural extension of this presentation to a resolution.

\begin{definition}
Let $\RR_\bullet^\cL(X^n)$ denote the following semi-simplicial set. The set of $p$-simplices $\RR_p^\cL(X^n)$ is the set of $p$-tuples of polytopes $(P_1,\ldots,P_p)$ with disjoint interiors. The face maps are given as follows:
\[d_i(P_1,\ldots,P_p) =
\begin{cases}
(P_2,\ldots,P_p), & \text{for } i=0,\\
(P_1,\ldots,P_i \cup P_{i+1},\ldots, P_p), & \text{for } 0<i<p,\\
(P_1,\ldots,P_{p-1}), & \text{for } i=p.\\
\end{cases}\]
Let $\RR^\cL(X^n)$ denote the geometric realization of $\RR_\bullet^\cL(X^n)$.
\end{definition}

Observe that $\widetilde{H}_0(\RR^\cL(X^n)) \cong 0$ and $\widetilde{H}_1(\RR^\cL(X^n)) \cong \smash{\Pt^\cL(X^n)}$. 

\begin{question} \label{resolutionQuestion}
Do the higher homology groups $\widetilde{H}_i(\RR^\cL(X^n))$ vanish for $i \geq 2$ when  $\cL$ is generated by hyperplanes and is admissible? 
\end{question}

If these homology groups vanish, then the cellular chains on $\RR^\cL(X^n)$ give a resolution of the polytope group in the style of the classical bar resolution. If they do not vanish, then the higher homology groups of $\RR^\cL(X^n)$ can be viewed as potentially interesting higher versions of the polytope group.
	
	\sloppy
	\printbibliography
\end{document}